\documentclass[reqno,makeidx,11pt]{amsart}

\usepackage[T1]{fontenc}
\usepackage[utf8]{inputenc}
\usepackage{subcaption}
\usepackage{graphicx}
\usepackage{latexsym}
\usepackage{amstext}
\usepackage {amsmath}
\usepackage {amsfonts}
\usepackage {amssymb}
\usepackage {amsthm}
\usepackage{color}
\usepackage {bbm}
\usepackage{esint}
\usepackage{array}
\usepackage{comment}
\usepackage{xspace}
\usepackage[bookmarks=true]{hyperref}
\usepackage{enumitem}
\usepackage{mathtools}
\graphicspath{{.}{}}
\ifx\cs\documentclass \footheight 12pt \fi \footskip 30pt
\makeatletter
\newcommand{\pushright}[1]{\ifmeasuring@#1\else\omit\hfill$\displaystyle#1$\fi\ignorespaces}
\newcommand{\pushleft}[1]{\ifmeasuring@#1\else\omit$\displaystyle#1$\hfill\fi\ignorespaces}
\makeatother

\DeclareMathOperator{\argmin}{argmin}
\DeclareMathOperator{\loc}{loc}

\DeclareMathAlphabet{\mathpzc}{OT1}{pzc}{m}{it}

\DeclareMathOperator\id{Id}

\DeclareMathOperator{\dive}{div}
\DeclareMathOperator{\diam}{diam}
\DeclareMathOperator{\genus}{genus}

\newcommand{\Chi}{\mathcal{X}}
\newcommand{\R}{\ensuremath{\mathbb{R}}}
\newcommand{\N}{\ensuremath{\mathbb{N}}}
\def\Z{\mathbb{Z}}

\newcommand{\F}{\ensuremath{\mathcal{F}}}

\newcommand{\Ha}{\ensuremath{\mathcal{H}}}

\def\W{W}
\def\H{\mathcal{H}}

\def\M{\mathcal{M}}
\def\calM{\mathcal{M}}

\def\lt{\left}
\def\rt{\right}
\def\les{\lesssim}
\def\ges{\gtrsim}
\def\eps{\varepsilon}
\def\Xint#1{\mathchoice
{\XXint\displaystyle\textstyle{#1}}%
{\XXint\textstyle\scriptstyle{#1}}%
{\XXint\scriptstyle\scriptscriptstyle{#1}}%
{\XXint\scriptscriptstyle\scriptscriptstyle{#1}}%
\!\int}
\def\XXint#1#2#3{{\setbox0=\hbox{$#1{#2#3}{\int}$ }
\vcenter{\hbox{$#2#3$ }}\kern-.57\wd0}}

\def\dashint{\Xint-}
\newcommand{\mres}{\mathbin{\vrule height 1.6ex depth 0pt width 0.13ex\vrule height 0.13ex depth 0pt width 1.3ex}}

\theoremstyle{plain}
\numberwithin{equation}{section}
\newtheorem{theorem}{Theorem}[section]
\newtheorem{corollary}[theorem]{Corollary}
\newtheorem{lemma}[theorem]{Lemma}
\newtheorem{proposition}[theorem]{Proposition}

\theoremstyle{definition}
\newtheorem{remark}[theorem]{Remark}
\newtheorem*{remark-short}{Remark}

\begin{document}
\title[Estimates for bending energies and non-local variational problems]{Quantitative estimates for bending energies and applications to non-local variational problems}
\author{Michael Goldman}
\address{Michael Goldman, Laboratoire Jacques-Louis Lions (CNRS, UMR 7598), Universit\'e Paris Diderot, F-75005, Paris, France}
\email{goldman@math.univ-paris-diderot.fr}

\author{Matteo Novaga}
\address{Matteo Novaga, Dipartimento di Matematica, Universit\`a di Pisa, 
Largo B. Pontecorvo 5, 56127 Pisa, Italy}
\email{matteo.novaga@unipi.it}

\author{Matthias R{\"o}ger}
\address{Matthias R{\"o}ger, Technische Universit\"at Dortmund, Fakult\"at f\"ur Mathematik, Vogelpothsweg 87, D-44227 Dortmund, Germany}
\email{matthias.roeger@tu-dortmund.de}



\date{\today}

\begin{abstract}
We discuss a variational model, given by a weighted sum of perimeter, bending and  Riesz interaction energies, that could be considered as a toy model for charged elastic drops. 
The different contributions have competing preferences for strongly localized and maximally dispersed structures. We investigate the energy landscape in dependence of the size of the `charge', i.e.~the weight of the Riesz interaction energy.

In the two-dimensional case we first prove that for simply connected sets of small elastica energy, the elastica deficit controls the isoperimetric deficit. 
Building on this result, we show that for small charge the only minimizers of the full variational model are either balls or centered annuli. 
We complement these statements by a non-existence result for large charge.
In three dimensions, we prove area and diameter bounds for configurations with small Willmore energy and show that balls are the unique minimizers of our variational model for sufficiently small charge.
\end{abstract}

\maketitle

\tableofcontents

\section{Introduction}
\label{sec:intro}
In recent years there has been a strong interest in variational models involving a competition between a perimeter type energy and a repulsive term of long-range nature
(see for instance the recent review papers \cite{ChMuTo,GolRuf} and the detailed discussion below). The aim of this paper is to start investigating the
 effects for this class of problems of higher order interfacial energies such as the Euler elastica in dimension two or the Willmore energy in dimension three. We will consider the simplest 
possible setting and study volume constrained minimization of functionals defined for sets $E\subset \R^d$ with $d=2,3$ as  

\begin{equation}
	\lambda P(E)+ \mu W(E)+QV_\alpha(E),\quad \lambda,\mu,Q\geq 0. \label{eq:fun-pre}
\end{equation}
The different contributions are given by
\begin{itemize}
\item the perimeter $P$, defined as
\begin{equation*}
	P(E) = \Ha^{d-1}(\partial E), 
\end{equation*}
\item the elastica or Willmore energy $W$, defined as
\begin{equation*}	W(E) = 
	\begin{cases}
		\displaystyle\int_{\partial E} H^2\,d\Ha^1 \quad&\text{ for }d=2,
		\\
		\\
		\displaystyle\frac{1}{4}\int_{\partial E} H^2\,d\Ha^2\quad&\text{ for }d=3,
	\end{cases}
\end{equation*}
where  $H$ denotes the mean curvature of $\partial E$ i.e.~the curvature in dimension two and the sum of the principal curvatures in dimension three\footnote{We choose to keep the 
factor $\frac{1}{4}$ in dimension three to stick with the traditional notation.},\\ 
\item the Riesz interaction energy $V_\alpha$, defined for $\alpha\in(0,d)$ as
\begin{equation*}
	V_\alpha(E) = \int_{E\times E} \frac{1}{|x-y|^{d-\alpha}}\,dx\, dy. 
\end{equation*}
\end{itemize}

For $\mu=0$ functional \eqref{eq:fun-pre} is arguably the simplest example of an isoperimetric type problem showing competition between a local attractive term with a non-local repulsive term.
In the case of Coulombic interactions, that is $d=3$ and $\alpha=2$, this model  appears in a variety of contexts. It is for instance the  celebrated Gamow liquid drop model for atomic nuclei \cite{Gamo30} or 
the sharp interface limit of the so-called Ohta-Kawasaki model for diblock copolymers \cite{OhtaKawasaki,BaFr99}. See also \cite{pasta} for another application of this model. 
Even though the picture is not complete, it has been shown that minimizers are balls for small $Q$ \cite{KnMu13,KnMu14,CiSp13,julin,F2M3} (actually they are the only stable critical sets \cite{julin2}) and  that no minimizers exist for large $Q$ \cite{KnMu13,KnMu14,LuOt14} (see also \cite{FrKN16} for a simple proof of non-existence in the three dimensional case).
Many more questions related to pattern formation have been investigated for very closely related  models, see for instance 
\cite{AlCO09, AcFuMo,ChPe10,ReWe08,BoCr14,GolRun,GMS,KnMN16} for a non-exhaustive list. Other examples of functionals presenting this type of competition can be found for instance in shape memory alloys \cite{KM,KKO}, 
micromagnetics \cite{DKMO} or epitaxial growth \cite{BGZ,Leoni}. However, the closest model is probably the one for charged liquid drops introduced in \cite{Ray} where the Riesz interaction energy is
replaced by a capacitary term. Surprisingly, it has been shown in \cite{GolNoRuf} that independently of the charge, no minimizers exist for this model (see also \cite{MurNov,MurNovRuf}). It has been suggested
in \cite{GolRuf} that a regularization by a Willmore type energy such as the one considered here might restore well-posedness. This paper can be seen as a first step in this direction. \\ 
The energy \eqref{eq:fun-pre} could be seen as a toy model for charged elastic vesicles, where the Willmore energy represents a prototype of more general bending energies for fluid membranes and a Coulomb self-interactions refers to the energy of a uniformly charged body. Associated with this interpretation, we refer to the parameter regimes $Q\ll 1$ and $Q\gg 1$ by the terms small and large `charge'.

Our goal is to understand how the picture changes for \eqref{eq:fun-pre} in the presence of a bending energy i.e.~for $\mu>0$.  For $d=2$ and fixed volume, since an annulus of 
large radius has arbitrary small elastica energy and also arbitrary small Riesz interaction energy, one needs to either restrict
the class of competitors to simply connected sets or to include the perimeter penalization (that is take $\lambda>0$). 
In contrast, for $d=3$, the Willmore functional is scaling invariant  and is globally minimized by   balls \cite{Simo83}. It seems then natural to consider \eqref{eq:fun-pre} for $\lambda=0$ and study the stability of the ball.
Let us point out that compared to the planar case, configurations with catenoid type parts allow for a much larger variety of structures with low energy. 
This makes the identification of optimal structures and the distinction between existence and non-existence of minimizers even more challenging.

\subsection*{Setting and main results}
Let us set some notation and give our main results. We will always assume that the energy \eqref{eq:fun-pre} contains the bending term and we set without loss of generality $\mu=1$.
For $\lambda,Q\geq 0$, we define
\begin{equation*}
	\F_{\lambda,Q} = \lambda P+W+QV_\alpha. 
\end{equation*}
Regarding the volume constraint, as will be better explained later on, up to a rescaling there is no loss of generality in assuming that $|E|=|B_1|$. 
For $d=2,3$, we define  the following classes of admissible sets
\begin{align*}
	\M\,&=\, \{ E\subset \R^d\text{ bounded with $W^{2,2}$-regular boundary}\},\\
	\M_{sc}\,&=\, \{ E\in\M\,:\, E\text{ simply connected}\},\\
	\M(|B_1|)\,&=\, \{ E\in\M\,:\, |E|=|B_1|\},\\
	\M_{sc}(|B_1|)\,&=\, \{ E\in\M_{sc}\,:\, |E|=|B_1|\},
\end{align*}
and consider the variational problems
\begin{equation}\label{problem2d}
 \min_{ \M_{sc}(|B_1|)} \F_{\lambda,Q}(E)
\end{equation}
and
\begin{equation}\label{problem2d-gen}
 \min_{ \M(|B_1|)} \F_{\lambda,Q}(E).
\end{equation}
We start by considering the planar problem $d=2$ and first focus on the uncharged case $Q=0$. For $\lambda=0$ no global minimizer exists in $\M(|B_1|)$, but it has been recently shown in \cite{BucurHenrot,FeKN16} 
that balls minimize the elastica energy under volume constraint in the class of simply connected sets. Our first result is a quantitative version of this fact in the spirit of the quantitative isoperimetric inequality \cite{FuMaPra,CicLeo}.
\begin{theorem}\label{theo:introstabball}
 There exists a universal constant $c_0>0$ such that for every set $E\in\M_{sc}(|B_1|)$, 
\begin{equation*}
	 W(E)-W(B_1)\,\ge\, 
 c_0 \min_{x} |E\Delta B_1(x)|^2,
\end{equation*}
where $E\Delta F$ denotes the symmetric difference of the sets $E$ and $F$.\\
Furthermore, there exist $\delta_0>0$ and  $c_1 >0$ such that if $W(E)\leq W(B_1)+\delta_0$, then
\begin{equation*}
 W(E)-W(B_1)\,\ge\, c_1 (P(E)-P(B_1)).
\end{equation*}
\end{theorem}
The proof is based on the idea of \cite{CicLeo} for the proof of the quantitative isoperimetric inequality to reduce by a contradiction argument to the case of nearly spherical sets and then compute a Taylor expansion along the lines of \cite{fuglede}.
As opposed to \cite{CicLeo} which is based on an improved convergence theorem, we obtain the strong convergence to the ball directly from the energy and a delicate refinement of \cite{BucurHenrot} (see Lemma \ref{lem:quantWill2}).\\

Still in the case $Q=0$, we then remove the constraint on the sets to be simply connected but consider the minimization problem \eqref{problem2d-gen} for $\lambda>0$.
\begin{theorem}\label{theo:minlambda}
 Let $Q=0$ and $d=2$. There exists $\bar \lambda>0$ such that for $\lambda\in (0, \bar \lambda)$, minimizers  of \eqref{problem2d-gen} are annuli while for
$\lambda>\bar \lambda$ they are balls. 
\end{theorem}
Next, we  turn to the stability estimates analogous to Theorem \ref{theo:introstabball}.
\begin{theorem}\label{thm:stab_intro_Qzero}
  Let $d=2$ and $\bar \lambda$ be given by Theorem \ref{theo:minlambda}. Then, there exists a universal constant $c_2>0$, such that  for any $E\in \M(|B_1|)$ and $\lambda>\bar \lambda$
\begin{equation*}
  \F_{\lambda,0}(E)-\F_{\lambda,0}(B_1)\ge c_2 (\lambda-\bar \lambda)\min_x |E\Delta B_1(x)|^2,
\end{equation*}
while for any  $\lambda_*>0$ there exists a constant $c(\lambda_*)>0$ such that for any $\lambda\in [\lambda_*,\bar\lambda]$
\begin{equation*}
 \F_{\lambda,0}(E)- \min_{\M(|B_1|)}\F_{\lambda,0}\ge c(\lambda_*) \min_{\Omega} |E\Delta \Omega|^2,
\end{equation*}
where the minimum is taken among all sets $\Omega\in \M(|B_1|)$ such that $\Omega$ is a ball or $\Omega$
is an annulus minimizing $\F_{\lambda,0}$ in $\M(|B_1|)$.
\end{theorem}

Then, we turn to the study of \eqref{problem2d} and \eqref{problem2d-gen} for $Q>0$. Regarding \eqref{problem2d} we prove the following.
\begin{theorem}\label{thm:intro_ball_min_sc_Qpos}
Let $d=2$. There exists $Q_0>0$ such that for $Q<Q_0$ and all $\lambda\geq 0$, balls are the only minimizers of \eqref{problem2d}.
\end{theorem}
The proof is a combination of Theorem \ref{theo:introstabball} and \cite{KnMu13}. As for \eqref{problem2d-gen}, we obtain a good understanding of part of the phase diagram (see Figure \ref{fig:phasediagram}).
\begin{figure}\begin{center}
 \resizebox{9.5cm}{!}{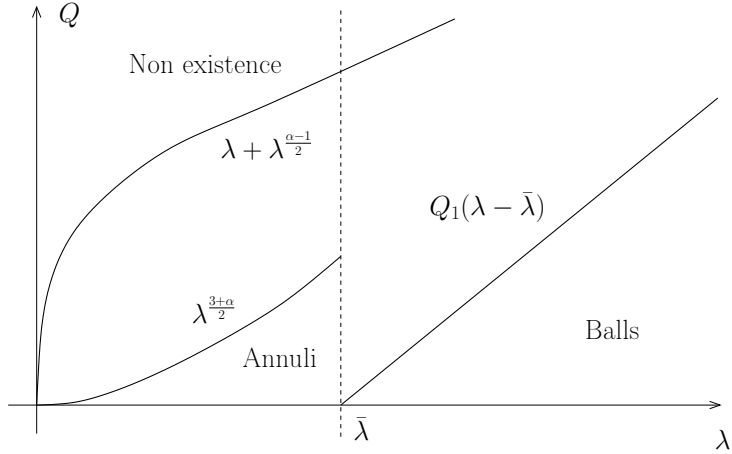}
   \caption{The  phase diagram.} \label{fig:phasediagram}
 \end{center}
 \end{figure}
\begin{theorem}\label{thm:intro_min_Qpos}
Let $d=2$.
 \begin{itemize}
  \item There exists $Q_1>0$ such that for every $\lambda>\bar \lambda$ and every $Q\le Q_1(\lambda-\bar \lambda)$, balls are the only minimizers of \eqref{problem2d-gen}.
  \item There exists $Q_2>0$ such that for every $\lambda\in(0, \bar \lambda]$ and every $Q\le Q_2 \lambda^{\frac{3+\alpha}{2}}$, centered annuli are the only minimizers of \eqref{problem2d-gen}.
  \item For every $\alpha\in(1,2)$ there exists $Q_3(\alpha)$ such that for every $\lambda\ge 0$ and every $Q\ge Q_3(\alpha)(\lambda+\lambda^{\frac{\alpha-1}{2}})$, no minimizer exists for \eqref{problem2d-gen}.
 \end{itemize}
\end{theorem}
The first part of the theorem is a direct consequence of the minimality of the ball for $\F_{\bar \lambda,0}$ and for $P+QV_\alpha$ for $Q$ small enough. 
The second point regarding the minimality of centered annuli is the most delicate part of the theorem. It requires first to argue that sets of small energy are almost annuli and
then to use the stability of annuli. The last part of the theorem regarding non-existence is obtained by noticing that if a minimizer exists then we can obtain a lower bound on the energy which is not compatible for large
$Q$ with  an upper bound obtained by constructing a suitable competitor. \\

We conclude the paper by studying the three dimensional case where 
a characterization of the energy landscape is even more difficult. As already pointed out, the Willmore energy is invariant under rescaling and is globally minimized by balls \cite{Will65,Simo83}.
We can thus focus on the case $\lambda=0$ where we have competition 
 between the Willmore energy and the Riesz interaction energy. Stability estimates for the Willmore energy have been obtained by De~Lellis and Müller \cite{DeMu05}. Building on
 these, on the control of the isoperimetric deficit by the Willmore deficit obtained in \cite{RoeSc12} and a bound on the perimeter (see Proposition \ref{prop:bounds}), we obtain that balls are minimizers of \eqref{problem2d-gen} for small $Q$.
 \begin{theorem}\label{theo:introball}
  For  $d=3$ and $\lambda=0$, there exists $Q_4>0$ such that for every $Q\le Q_4$, the only minimizers of \eqref{problem2d-gen} are balls.
 \end{theorem}
Of course, since balls are also minimizers of the isoperimetric problem, a direct consequence of Theorem \ref{theo:introball} is the minimality of the balls for \eqref{problem2d-gen} for every $\lambda\ge0$ and every $Q\le Q_4$.
 For the case $\lambda=0$ we are not able to prove or disprove a non-existence regime in the parameter space. Still,  we  show that if a minimizer exists for every $Q$ then its  isoperimetric quotient must degenerate as $Q\to \infty$ (see Proposition \ref{prop:min-large-m}).
  This is somewhat reminiscent of earlier results obtained by Schygulla \cite{Schy12}. 
  
 Finally, we obtain a non-existence result in the case $\lambda>0$ and $\alpha\in (2,3)$ in the regime of sufficiently large charge.
\begin{proposition}\label{prop:nonexistence-3D-intro}
 For every $\alpha\in (2,3)$, there exists $Q_5(\alpha)$ such that for every $\lambda, Q$ with $Q\ge Q_5(\lambda^{-\frac{3-\alpha}{2}}+\lambda^{\frac{3+\alpha}{2}})$, 
  no minimizer of
 $\F_{\lambda,Q}$ in $\M(|B_1|)$ exists. 
\end{proposition}
\subsection*{Outline and notation}
In Section \ref{sec:uncharged} we first consider the planar case $d=2$ in the absence of charge ($Q=0$) before considering the case $Q>0$ in Section \ref{sec:planarcharged}.
 In the last section we finally consider the three dimensional case. For the reader's convenience, the main theorems given in the introduction are restated in the respective sections and some of them have been extended by more detailed statements. Theorem \ref{theo:introstabball}, Theorem \ref{theo:minlambda}, and Theorem \ref{thm:stab_intro_Qzero} correspond to Theorem \ref{thm:quantWill}, Theorem \ref{theo:minFlambda} and Theorem \ref{thm:stab}, respectively. Theorem \ref{thm:intro_ball_min_sc_Qpos} corresponds to Proposition \ref{prop:simply_connected}, the statements in Theorem \ref{thm:intro_min_Qpos} are collected from Proposition \ref{prop:stabball}, Theorem \ref{theo:stabannulus}, Proposition \ref{prop:stabannulus2} and Proposition \ref{prop:non-existence}. Theorem \ref{theo:introball} corresponds to Theorem \ref{teomin}, and Proposition \ref{prop:nonexistence-3D-intro} to Proposition \ref{prop:nonexistence-3D}.

For two real numbers $A,B$ the notation $A\ges B$ means that $A\geq cB$ for some $c>0$ that is universal (unless dependencies are explicitly stated). 
Correspondingly we use the notation $\les$ and write $A\sim B$ if $A\les B$ and $A\ges B$. 

%
\section{The planar case: uncharged drops}\label{sec:uncharged}
We start by investigating the planar case $d=2$ in the absence of charges i.e.~$Q=0$. Our aim is both to characterize the minimizers and to show that the energy controls the distance to these minimizers. 
\subsection{Simply connected sets: controlling the asymmetry index by the elastica deficit}
We first restrict ourselves to simply connected sets. By \cite{BucurHenrot, FeKN16} balls are the unique minimizers of the elastica energy among simply connected sets with prescribed volume. Since they also minimize the perimeter, balls are the unique minimizers  of $\F_\lambda=\F_{\lambda,0}$ in this class. 
Using the quantitative isoperimetric inequality \cite{FuMaPra}, one could directly obtain a quantitative inequality for $\F_\lambda$ which would however degenerate as $\lambda\to 0$. Our aim here is to show that actually the elastica energy
$W(E)$ itself controls the distance to balls. This is a quantitative version of \cite{BucurHenrot, FeKN16} which could be of independent interest.

Inspired by a strategy first used in \cite{AcFuMo,CicLeo}  (see for instance also \cite{BraDeVe,CicLeoMa,GolNoRuf} for a few other applications)  which was building on ideas from \cite{fuglede}, we first restrict ourselves to nearly spherical sets. 
More precisely, we consider  sets $E$ such that $\partial E$ is a graph over $\partial B_R$ for some $R>0$, i.e.
\begin{equation}
	\partial E = \{R(1+\phi(\theta))e^{i\theta} \ : \ \theta\in[0,2\pi)\}, \label{eq:nearly-spherical}
\end{equation} 
with $\|\phi\|_{W^{2,2}}\ll 1$. We will need the following estimate on the elastica energy of nearly spherical sets.

\begin{lemma}
 Let $R>0$ and  $E$ be a nearly spherical set. Then,
 \begin{equation}\label{quantWillspherprep}
  W(E)-W(B_R)= R^{-1}\int_0^{2\pi}\Big(\ddot{\phi}^2+\phi^2+\frac{3}{2}\dot{\phi}^2 -\phi+4\phi \ddot{\phi}\Big) + o(\|\phi\|^2_{W^{2,2}}).
 \end{equation}
 Moreover, if $|E|=|B_R|$ and the barycenter of $E$ is equal to zero, then
 \begin{equation}\label{quantWillspher}
  W(E)-W(B_R)\ges R^{-1}\,\int_0^{2\pi} \Big(\ddot{\phi}^2+\dot{\phi}^2+\phi^2\Big).
 \end{equation}
\end{lemma}

\begin{proof}
 By scaling, it is enough to prove \eqref{quantWillspherprep}, \eqref{quantWillspher} for $R=1$. 
The elastica energy of $E$ is given by
\begin{equation*}
   \int_{\partial E} H^2=\int_{0}^{2\pi} \frac{(2\dot{\phi}^2+(1+\phi)^2-(1+\phi)\ddot{\phi})^2}{(\dot{\phi}^2+(1+\phi)^2)^{5/2}}.
\end{equation*}
Let us now compute the Taylor expansion of the integrand.  Keeping only up to quadratic terms, we get that on the one hand,
\begin{equation*}
   (2\dot{\phi}^2+(1+\phi)^2-(1+\phi)\ddot{\phi})^2=1+4\phi -2\ddot{\phi}+6\phi^2+\ddot{\phi}^2+4\dot{\phi}^2-6\phi \ddot{\phi} +(\phi^2+\dot{\phi}^2+ \ddot{\phi}^2)O(|\phi|+|\dot{\phi}|),
\end{equation*}
and on the other hand,
\begin{equation*}
   (\dot{\phi}^2+(1+\phi)^2)^{-5/2}=1-5\phi-\frac{5}{2}\dot{\phi}^2+15 \phi^2 +o(\phi^2+\dot{\phi}^2),
\end{equation*}
so that 
\begin{equation*}
  \frac{(2\dot{\phi}^2+(1+\phi)^2-(1+\phi)\ddot{\phi})^2}{(\dot{\phi}^2+(1+\phi)^2)^{5/2}}= 1-\phi-2\ddot{\phi}+\ddot{\phi}^2+\phi^2+\frac{3}{2}\dot{\phi}^2+4\phi \ddot{\phi} +(\phi^2+\dot{\phi}^2+ \ddot{\phi}^2)O(|\phi|+|\dot{\phi}|).
\end{equation*}
Using that  $\int_0^{2\pi} \ddot{\phi}=0$, we obtain \eqref{quantWillspherprep}.\\
If now $|E|=|B_1|$, using that $2|E|=\int_{\partial E}x\cdot \nu(x)\,d\Ha^1(x)$, we have
\begin{equation}\label{volumeconstraint}
	\int_{0}^{2\pi} \Big(\phi+\frac{\phi^2}{2}\Big)=0,
\end{equation}
while if the barycenter $\bar x$ is zero, using that $3|E| \bar x=\int_{\partial E}x\cdot\nu(x) x\,d\Ha^1(x)$,
 \begin{equation}\label{barycenter}
 \int_{0}^{2\pi} [(1+\phi)^3-1] e^{i\theta}=0.
\end{equation}
Using \eqref{volumeconstraint} in \eqref{quantWillspherprep}, we get 
\begin{equation}
   W(E) - W(B_1) = \int_{0}^{2\pi} \Big(\ddot{\phi}^2+\frac{3}{2}\phi^2+\frac{3}{2}\dot{\phi}^2+4\phi \ddot{\phi}\Big)+o(\|\phi\|^2_{W^{2,2}}). \label{eq:12a}
\end{equation}
If $\phi=\sum_{k\in \Z} a_{k} e^{ik\theta}$ is the Fourier representation of $\phi$, then \eqref{volumeconstraint} and \eqref{barycenter} imply that
\[
 |a_0|+|a_{\pm 1}|\les \sum_{|k|\ge 2} |a_k|^2
\]
and thus for every $j\in \N$,
\begin{equation*}
   \left\|\frac{d^j\phi}{d\theta^j}\rt\|_{L^2}^2\leq C \sum_{|k|\ge 2} |k|^{2j} |a_k|^2.
\end{equation*}
Since by Parseval's identity,
\begin{equation*}
   \int_{0}^{2\pi} \Big(\ddot{\phi}^2+\frac{3}{2}\phi^2+\frac{3}{2}\dot{\phi}^2+4\phi \ddot{\phi}\Big) = 2\pi\sum_{k\in \Z} \Big(|k|^4-\frac{5}{2}|k|^2+\frac{3}{2}\Big)|a_k|^2,
\end{equation*}
and since for $k\in \Z$, the polynomial $k^4-\frac{5}{2}k^2+\frac{3}{2}$ is always non-negative and vanishes only for $|k|=1$, we have for all $\phi$ with $\|\phi\|_{W^{2,2}}$ sufficiently small
\begin{align*}
 W(E)-W(B_1)&\ges \sum_{|k|\ge 2} \Big(|k|^4-\frac{5}{2}|k|^2+\frac{3}{2}\Big)|a_k|^2\\
 &\ges \int_{0}^{2\pi} \Big(|\ddot{\phi}|^2+|\dot{\phi}|^2+|\phi|^2\Big),
\end{align*}
concluding the proof of \eqref{quantWillspher}.
\end{proof}

We also recall the following Taylor expansion of the perimeter for nearly spherical sets (see \cite{fuglede}).

\begin{lemma}
 Let $R>0$ and let $E$ be a nearly spherical set with $\partial E$ represented as in \eqref{eq:nearly-spherical}, then
 \begin{equation}\label{quantpernearlyspher}
  P(E)-P(B_R)=R\int_0^{2\pi}  \Big(\phi +\frac{\dot{\phi}^2}{2}\Big) +o(\|\phi\|_{W^{1,2}}^2).
 \end{equation}
\end{lemma}

We now combine  estimate \eqref{quantWillspher} and the work of Bucur-Henrot \cite{BucurHenrot}, to obtain a  quantitative estimate on the elastica deficit. 
\begin{theorem}\label{thm:quantWill}
There exists a universal constant $c_0>0$ such that for every $R>0$ and every set $E\in\M_{sc}(|B_R|)$, 
\begin{equation}\label{quantWill}
	R\Big( W(E)-W(B_R)\Big)\,\ge\, 
 c_0 \lt(\min_{x\in \R^2} \frac{|E\Delta B_R(x)|}{|B_R|}\rt)^2.
\end{equation}
Furthermore, there exist $\delta_0>0$ and  $c_1 >0$ such that if $W(E)\leq R^{-1}(W(B_1)+\delta_0)$, then
\begin{equation}\label{quantWill2}
 R\Big(W(E)-W(B_R)\Big)\,\ge\, c_1 \frac{P(E)-P(B_R)}{P(B_R)}.
\end{equation}
\end{theorem}

In order to prove  this theorem we will need  some auxiliary results. Even for simply connected sets, uniform bounds on the volume and the elastica energy are in general not sufficient 
to obtain a perimeter or a diameter control (see for example \cite[Figure 1]{BucurHenrot}). However, this is the case for sets with  elastica energy  sufficiently small.
\begin{lemma}\label{lem:quantWill}
There exist $\delta_0>0$ and $C_0>0$ such that for every $E\in\M_{sc}(|B_1|)$ with $W(E)\leq W(B_1)+\delta_0$, there holds
\begin{equation}
	P(E) \,\leq\, C_0. \label{eq:perimeter-bound}
\end{equation}
\end{lemma}
\begin{proof}
We first prove, by contradiction, a corresponding bound for the diameter, and then deduce the perimeter bound.
\begin{enumerate}[align=left, leftmargin=0pt, 
listparindent=\parindent, labelwidth=0pt, itemindent=!,label=\underline{Step \arabic*}.]
\item
Assume that there exists a sequence $(E_n)_{n\in\N}$ in $\M_{sc}(|B_1|)$ with 
\begin{equation*}
	W(E_n)\,\to\, W(B_1)\quad\text{ and }\quad d_n=\diam(E_n)\to\infty. 
\end{equation*}
In the following steps, the implicit constants in $\les$ and $\ges$ estimates may depend on the fixed sequence $(E_n)_n$ but are independent of $n\in\N$.

First, up to a rotation, we may assume that for every $\xi\in(0,d_n)$ the vertical section $E^n_{\xi}=E_n\cap (\{\xi\}\times \R)\neq \emptyset$ is not empty. By Fubini's Theorem, there exists $\xi_n\in(\frac{d_n}{3},\frac{2d_n}{3})$ such that $|E^n_{\xi_n}|\les d_n^{-1}$. Then, there also exist $x_n\in \partial E_n\cap E^n_{\xi_n}$ 
and $y_n\in \partial E_n\cap E^n_{\xi_n}$ such that $(x_n,y_n)\subset E_n$ and $|x_n-y_n|\les  d_n^{-1}$.
\item
Choose an oriented tangent field $\tau$ on $\partial E_n$. We claim that, for some $C_1>0$ independent of $n$, we have
\begin{equation}\label{hyptang}
	\big|\tau(x_n)^\perp \cdot \tau(y_n)\big| \leq C_1d_n^{-1/3}.
\end{equation} 
Indeed, assume that \eqref{hyptang} does not hold so that there exists $\Lambda_n\to \infty$ with 
\begin{equation}\label{abshyptang}
 \big|\tau(x_n)^\perp \cdot \tau(y_n)\big| \geq \Lambda_n d_n^{-1/3}.
\end{equation}
Without loss of generality, using another translation and rotation, we can assume that $x_n=0$ and $\tau(x_n)=e_1$ (see Figure \ref{fig:intersect}). 
\begin{figure}[!h]
\centering
\resizebox{0.7\textwidth}{!}{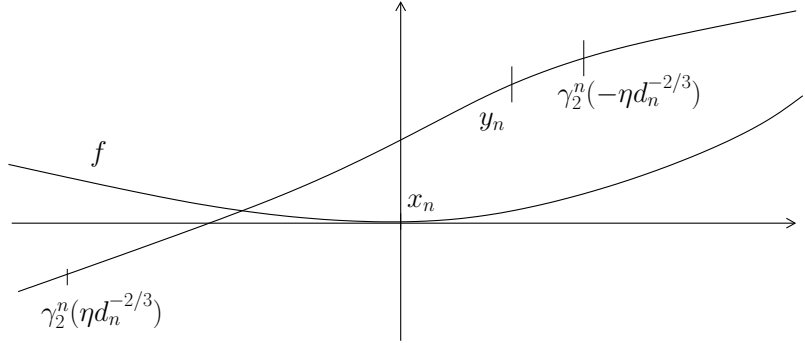}
\caption{Geometry of $\partial E_n$ around $x_n$ under assumption \eqref{abshyptang}.}\label{fig:intersect}
\end{figure}

By the bound on the elastica energy we can locally write the part of $\partial E_n$ containing $0$ as a graph of the form $\{(\xi,f(\xi)) : |\xi|\le d_n^{-2/3}\}$ with $f(0)=f'(0)=0$ and uniformly bounded slope (see \cite[Lemma 2.1]{BucurHenrot}). 
Moreover, $\sup_{|\xi|\le d_n^{-2/3}} |f(\xi)|\leq C d_n^{-1}$, since for $0\le\xi\le d_n^{-2/3}$
\begin{align*}
	|f(\xi)|\,&\leq\, \int_0^{\xi} (\xi-s)|f''(s)|\,ds \\
	&\leq\, \Big(\int_0^\xi \frac{(f'')^2}{(1+(f')^2)^{\frac{5}{2}}}\Big)^{1/2}\Big(\int_0^\xi (\xi-s)^2 (1+(f')^2)^{\frac{5}{2}}\Big)^{1/2}\\
	&\leq\,C \sqrt{W(E_n)}\lt(1+\sup_{|\xi|\le d_n^{-2/3}}|f'|^{\frac{5}{2}}\rt)\xi^{\frac{3}{2}},
\end{align*}
and similarly for $-d_n^{-2/3}\le \xi\le 0$.
Let $\gamma^n=(\gamma_1^n,\gamma_2^n)$ be an arclength parametrization of $\partial E_n$. We may assume that $y_n=\gamma^n(0)$ with $|\gamma^n(0)|\les d_n^{-1}$. 
By \eqref{abshyptang}, $|\dot{\gamma}^n_2(0)|=\big|\tau(x_n)^\perp \cdot \tau(y_n)\big|\ge  \Lambda_n d_n^{-1/3}$. Let us assume for definiteness  that $\dot{\gamma}^n_2<0$ (the other case being analogous). We then have by the bound on the elastica energy that for every $|t|\le d_n^{-2/3}$, 
\begin{equation*}
  \dot{\gamma}_2^n(t)=\dot{\gamma}_2^n(0)+\int_0^t \ddot{\gamma}_2^n ds\le -\Lambda_n d_n^{-1/3} + t^{1/2} \lt(\int_0^t (\ddot{\gamma}_2^n)^2\rt)^{1/2} ds\le -\frac{\Lambda_n}{2} d_n^{-1/3}. 
\end{equation*}

Let $\eta$ be a small constant  chosen so that for $t\in[-\eta d_n^{-2/3}, \eta d_n^{-2/3}]$,
\begin{equation*}
  |\gamma_1^n(t)|=\lt|(y_n)_1+\int_0^t \dot{\gamma}_1^n(s)ds\rt|\les d_n^{-1}+ \eta d_n^{-2/3}\le d_n^{-2/3}.
\end{equation*}
This implies that for $t\in[-\eta d_n^{-2/3}, \eta d_n^{-2/3}]$, $\gamma^n(t)$ stays inside the cylinder $[-d_n^{-2/3},d_n^{-2/3}]\times \R$. Furthermore,
\begin{equation*}
   \gamma_2^n(-\eta d_n^{-2/3})=(y_{n})_2 - \int_{-\eta d_n^{-2/3}}^0 \dot{\gamma}_2^n(s) ds \ges -d_n^{-1}+ \Lambda_n d_n^{-1}> \sup_{|\xi|\le d_n^{-2/3}} |f(\xi)|,
\end{equation*}
and 
\begin{equation*}
   \gamma_2^n(\eta d_n^{-2/3})=(y_{n})_2+\int_0^{\eta d_n^{-2/3}} \dot{\gamma}_2^n(s) ds \les d_n^{-1}- \Lambda_n d_n^{-1}< -\sup_{|\xi|\le d_n^{-2/3}} |f(\xi)|.
\end{equation*}
Therefore, the graph $\{(\xi,f(\xi)) \ : \ |\xi|\le d_n^{-2/3}\}$ splits the cylinder $[-d_n^{-2/3},d_n^{-2/3}]\times \R$ into two connected components with $\gamma_2^n(-\eta d_n^{-2/3})$ in one of the components and $\gamma_2^n(\eta d_n^{-2/3})$ in the other (see Figure \ref{fig:intersect}). 
\begin{figure}[!h]
\centering
\resizebox{0.6\textwidth}{!}{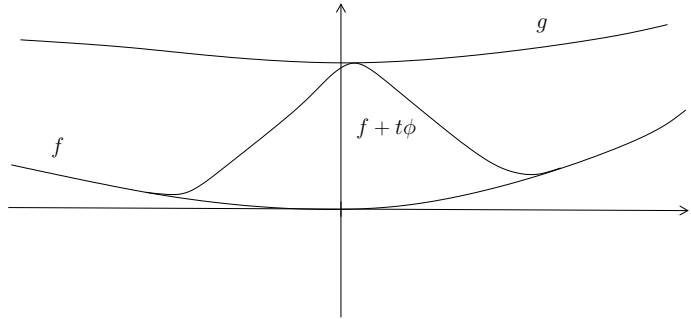}
\caption{Construction of the deformed set.}\label{fig:deform}
\end{figure}
Hence, 
$\gamma^n$ intersects the graph of $f$ which contradicts the fact that $\partial E_n$ can be locally written as a graph. We have thus shown that \eqref{hyptang} holds.
\item
We recall that we assume (without loss of generality) that $x_n=0$ and $\tau(x_n)=e_1$. By the bound on the elastica energy and  \eqref{hyptang}, for $\eta$ small enough (but not depending on $n$), $x_n$ and $y_n$ belong to two different connected components of
$\partial E_n \cap Q_\eta$,  where $Q_\eta=(-\eta,\eta)^2$. Let  $\Gamma_1$, respectively $\Gamma_2$ be the connected component containing $x_n$, respectively $y_n$. 
By \cite[Lemma 2.1]{BucurHenrot} and \eqref{hyptang}, we have $\Gamma_1\cap Q_\eta=\{(\xi,f(\xi)), |\xi|\le \eta\}$ and $\Gamma_2\cap Q_\eta=\{(\xi,g(\xi)), |\xi|\le \eta\}$ where without loss of generality, we can assume that $f(\xi)<g(\xi)$ for $|\xi|\le \eta$ (see Figure \ref{fig:deform}).     
Let now $\phi\in C^\infty(-\eta,\eta)$ be a non-negative bump function with $\phi(0)=1$ and let $d_n^{-1}\ges t>0$ be such that 
\begin{equation*}
   \max_{[-\eta,\eta]} \big((f+t\phi)-g\big) = 0. 
\end{equation*}
If we replace in $Q_\eta$ the component $\Gamma_1$ by $\widetilde{\Gamma}_1=\{ (\xi,f(\xi)+t\phi(\xi)) \ : \ |\xi|\le \eta\}$,  we obtain a new set $\widetilde{E}_n$ with $|\widetilde{E}_n|\le |E_n|$ (since by construction $\widetilde{E}_n\subset E_n$) and 
\begin{equation*}
   |W(\widetilde{E}_n)-W(E_n)|\le \lt|\int_{-\eta}^{\eta} \frac{(f''+t\phi'')^2}{(1+(f'+t\phi')^2)^{5/2}}- \frac{f''^2}{(1+f'^2)^{5/2}}\rt|\les t\les d_n^{-1},
\end{equation*}
where we have used that thanks to the energy estimate, $f'$ is uniformly small in $(-\eta,\eta)$ to make the  Taylor expansion. The set $\widetilde{E}_n$ is made of two drops  $E_n^1$ and $E_n^2$ with mass $m_1^n=|E_n^1|$ and $m_2^n=|E_n^2|$ satisfying $m_1^n+m_2^n\le |E_n|$. From \cite[Theorem 3.5]{BucurHenrot}, for every couple of drops $E$ and $F$, with $|E|+|F|=|B_{2^{-1/3}}|$,
\begin{equation*}
   W(F)+W(E)\ge (1+\delta_*) W(B_{2^{-1/3}}),
\end{equation*}
for some $\delta_*>0$.
By scaling, we deduce that if we choose a ball $B$ such that $|B|=|m_1^n|+|m_2^n|\le |E|=|B_1|$,  then
\begin{equation*}
    W(\widetilde{E}_n)=W(E^1_n)+W(E^2_n)\ge (1+\delta_*) W(B)\ge (1+\delta_*) W(B_1),
\end{equation*}
from which we obtain that 
\begin{equation*}
  W(E_n)\ge  W(\widetilde{E}_n)-Cd_n^{-1}\ge \lt(1+\frac{\delta_*}{2}\rt) W(B_1),
\end{equation*}
contradicting the fact that $W(E_n)\to W(B_1)$.
\item
By the previous steps we know that there exists $R>0$ such that $\diam(E)\leq 2R$ for all $E\in\M_{sc}(|B_1|)$ with $W(E)\leq W(B_1)+\delta_0$. Therefore, after translation $E\subset {B_R}$. We now choose an arclength parametrization $\gamma:[0,L]\to \R^2$, $L=P(E)$ and obtain as in \cite[Lemma 2.5]{BucurHenrot}
\begin{align*}
	L\,=\, \int_0^L |\dot\gamma|^2\,ds \,=\, -\int_0^L \gamma\cdot \ddot\gamma\,ds \,\leq\,\Big(\int_0^L|\gamma|^2\,ds\Big)^{\frac{1}{2}} \sqrt{W(E)}\,\leq\, L^{\frac{1}{2}}R(2\pi+\delta_0)^{\frac{1}{2}},
\end{align*}
from which the perimeter bound follows.
\end{enumerate}
\end{proof}

\begin{lemma}\label{lem:quantWill2}
Let $(E_n)_n$ be a sequence in $\M_{sc}(|B_1|)$ with $W(E_n)\to W(B_1)$ as $n\to\infty$. For every $n\in\N$ let $\gamma_n:[0,2\pi)\to\R^2$ be a constant speed parametrization of $\partial E_n$. Then, 
after translation $\gamma_n$ converges  strongly in $W^{2,2}$ to a (unit speed) arclength parametrization of $\partial B_1$. 
\end{lemma}
\begin{proof}
Consider the sets $\widetilde E_n= 2^{-\frac{1}{3}}E_n$. It follows from \cite{BucurHenrot} and our assumptions that $(\widetilde E_n)_n$ is a minimizing sequence of the 
functional $E\mapsto |E|+\frac{1}{2}W(E)$ on $\M$. Moreover, $(\widetilde E_n)$ has uniformly bounded perimeter by Lemma \ref{lem:quantWill}. For such sequences it is proved in 
\cite[Section 4]{BucurHenrot} that, after translation, $\widetilde E_n$ must converge to $B_{2^{-\frac{1}{3}}}$. Hence $E_n$ converges to $B_1$. This gives first weak convergence in $W^{2,2}$ which combined with
the convergence of the energy gives the strong convergence.
\end{proof}

\begin{proof}[Proof of Theorem \ref{thm:quantWill}]
Without loss of generality, we may assume that $R=1$ and $|E|=\pi$. 
\begin{enumerate}[align=left, leftmargin=0pt, 
listparindent=\parindent, labelwidth=0pt, itemindent=!,label=\underline{Step \arabic*}.]
\item
Assume for the sake of contradiction that \eqref{quantWill} does not hold. Then, there exists a sequence $(E_n)_n$ such that 
\begin{equation}\label{hypabsurdW}
 \frac{W(E_n)-W(B_1)}{\min_x |E_n\Delta B_1(x)|^2}\to 0.
\end{equation}
Since $\min_x |E_n\Delta B_1(x)|$ is bounded, this implies that $W(E_n)\to W(B_1)$ as $n\to\infty$. We then obtain from Lemma \ref{lem:quantWill} that the diameter $d_n$ of $E_n$ remains bounded.
Hence, by Lemma \ref{lem:quantWill2} $E_n$ must converge up to translation to $B_1$ strongly in $W^{2,2}$ and thus by Sobolev embedding, also in  $C^{1,\alpha}$ for every $\alpha\le 1/2$. Thus, for $n$ large enough, $\partial E_n$ is a graph over $B_1$ and $\partial E_n$ is nearly spherical.
Since  the barycenter of $E_n$ is converging to zero, $E_n$ is also a graph over the ball centered in its barycenter. Up to a translation, this means that we can apply
\eqref{quantWillspher} and  obtain   that
\begin{equation}
  W(E_n)-W(B_1)\ges\int_0^{2\pi} \phi_n^2\ges |E_n \Delta B_1|^2, \label{eq:quantWill-pf}
\end{equation}
which contradicts \eqref{hypabsurdW}.
\item
We now turn to \eqref{quantWill2}. The additional assumption implies in the rescaled setting that $W(E)\leq 2\pi+\delta_0$.  If \eqref{quantWill2} does not hold there exists a sequence $(E_n)_n$ such that
\begin{equation}\label{hypabsurdW-1}
 \frac{W(E_n)-W(B_1)}{P(E_n)-P(B_1)}\to 0.
\end{equation}
In particular, since by Lemma \ref{lem:quantWill}, $P(E_n)$ is uniformly bounded, we have as above that $W(E_n)\to W(B_1)$ as $n\to\infty$. Using \eqref{volumeconstraint} and \eqref{quantpernearlyspher} we deduce as in \eqref{eq:quantWill-pf} that 
\begin{equation*}
  W(E_n)-W(B_1)\ges  \int_0^{2\pi} \Big(\dot{\phi}_n^2+\phi_n^2\Big) \ges P(E_n)-P(B_1)
\end{equation*}
 which is in contradiction with \eqref{hypabsurdW-1}.
\end{enumerate}
\end{proof}
\subsection{Minimizers of $\F_\lambda$}
We move on and remove the constraint that sets are simply connected. Since annuli of very large diameter have vanishing elastica energy, in order  to have a well-posed problem, 
we need to consider $\lambda>0$.
Let us recall that the energy is given by
\begin{equation*}
  	\F_\lambda(E) = \lambda P(E)+W(E) = \int_{\partial E} \big(\lambda +H^2\big)\,d\Ha^1.
\end{equation*}
Up to a rescaling, we may restrict ourselves to the problem
\begin{align}
 \min_{\M(|B_1|)} &\F_\lambda(E).  \label{probisoper}
\end{align}
 We will show below that depending on the value of $\lambda$, minimizers are either balls or annuli. 
Before stating the precise result we compute the energy of an annulus. 

\begin{lemma}\label{rem:annuli}
For every $r>0$, the energy $\F_{\lambda}$ of an annulus with inner radius $r>0$ and area equal to $\pi$ is given by
\begin{equation}
	f_\lambda(r) = 2\pi\lt[ \lambda( r+(1+r^2)^{1/2})+\frac{1}{r}+\frac{1}{(1+r^2)^{1/2}}\rt]. \label{eq:flambda}
\end{equation}
The function $f_\lambda:(0,\infty)\to\R$ is strictly convex and $f^{(3)}_\lambda<0$. Its unique minimizer $r_\lambda$ is the unique solution of 
\begin{equation}\label{ELannulus}
  \lambda \lt(1+ \frac{r}{(1+r^2)^{1/2}}\rt)-\frac{1}{r^2}-\frac{r}{(1+r^2)^{3/2}}=0.
\end{equation}
\end{lemma}
\begin{proof}
The strict convexity of $f_\lambda$ follows from
\begin{equation*}
  \lt[r^{-1}+(1+r^2)^{-1/2}\rt]''=2 r^{-3} +(1+r^2)^{-5/2}(2r^2-1)\ge 2 r^{-3} -(1+r^2)^{-5/2}>0,
\end{equation*}
which can be checked by considering separately the case $r\le 1$ and the case $r\ge 1$. Since $f_\lambda(r)\to\infty$ as $r\to 0$ and $r\to\infty$ 
it has a unique minimizer $r_\lambda>0$ which then satisfies \eqref{ELannulus}. We further compute that 
\begin{equation*}
	f_\lambda^{(3)}(r) = - \frac{3 \lambda r}{\left(r^{2} + 1\right)^{\frac{5}{2}}} - \frac{15 r^{3}}{\left(r^{2} + 1\right)^{\frac{7}{2}}} + \frac{9 r}{\left(r^{2} + 1\right)^{\frac{5}{2}}} - \frac{6}{r^{4}}\leq 3\frac{-2r^3+3r-10}{(r^2+1)^\frac{7}{2}}<0.
\end{equation*}
\end{proof}

We  may now solve the minimization problem \eqref{probisoper}. 
\begin{theorem}\label{theo:minFlambda}
There exists $\bar \lambda>0$ such that for $\lambda\in(0, \bar \lambda)$, minimizers  of \eqref{probisoper} are annuli of inner radius $r_\lambda$ (as defined in Lemma \ref{rem:annuli}) while for
$\lambda>\bar \lambda$ they are balls of radius one. Moreover, for $\lambda= \bar \lambda$  minimizers can be either a ball  of radius one or an annulus of inner radius $r_{\bar \lambda}$.
\end{theorem}

\begin{proof}
For an arbitrary set $E\in \M(|B_1|)$, by translation invariance, we can in addition assume that the connected components of $E$ are far apart from each other so that in particular 
their convex envelopes do not intersect. In the first two steps we prove that minimizers must be either balls or annuli. In the final step we compare the minimal
energy of the optimal annuli with the energy of the ball to conclude the proof.
\begin{enumerate}[align=left, leftmargin=0pt, listparindent=\parindent, labelwidth=0pt, itemindent=!,label=\underline{Step \arabic*}.]
\item
Let $E$ be an admissible set. We first show that the energy of  each connected component $F$ of $E$ can be strictly decreased by transforming it into  a ball or an annulus. 
If $F$ is simply connected, by \cite{BucurHenrot,FeKN16} and the isoperimetric inequality 
\begin{equation*}
	\F_\lambda(F)\ge \F_\lambda(B_R),
\end{equation*}
for $R>0$ with $|B_R|=|F|$. Equality holds if and only if $F$ is a translate of $B_R$, proving the claim in this case. 

If $F$ is not simply connected, then $F^c$ is made of an unbounded connected component and a finite union  $G_1, \dots, G_n$ of 
 bounded simply connected sets. For $i=1,\dots n$, let $m_i=|G_i|$. By the discussion above, among all simply connected sets of mass $m_i$, $\F_\lambda$ is uniquely minimized by balls $B_i$ of area $m_i$. 
 For two balls $B_i$ and $B_j$, if $r>0$ is such that $|B_r|=|B_i|+|B_j|$, then 
 \[\F_\lambda(B_i)+\F_\lambda(B_j)>\F_\lambda(B_r)
\]
 since the inequality  is separately true for the perimeter and the elastica parts of the energy. Letting thus $r>0$ be such that $|B_r|=\sum_{i=1}^n m_i$, we have
 \begin{equation*}
  \F_\lambda(\cup_i G_i)=\sum_i \F_\lambda(G_i)\ge \sum_i \F_\lambda(B_i)\ge \F_\lambda(B_r)
 \end{equation*}
with equality if and only if $n=1$ and $G_1$ is a translate of $B_r$. The   set $F\cup \big(\cup_{i=1}^n G_i\big)$  given by filling the holes of $F$
is simply connected and letting 
$R>r$ be such that $|B_R|=|F|+|B_r|=|F\bigcup \cup_{i=1}^n G_i|$, we have as above that 
\begin{equation*}
   \F_\lambda(F\bigcup \cup_{i=1}^n G_i)\ge \F_\lambda(B_R)
\end{equation*}
with equality if and only if $F\bigcup \cup_{i=1}^n G_i$ is a translate of $B_R$. Putting all this together, we find that 
\begin{equation*}
   \F_\lambda(F)= \F_\lambda(F\bigcup \cup_{i=1}^n G_i)+ \F_\lambda(\cup_i G_i)\ge \F_\lambda(B_R)+\F_\lambda(B_r)=\F_\lambda(B_R\backslash B_r)
\end{equation*}
with equality if and only if $F$ is a (not necessarily concentric) annulus of outer radius $R$ and inner radius $r$, which again proves our claim.
\item
We are thus reduced to the class of competitors made of a finite union of balls and annuli. Since the elastica energy (respectively the perimeter) blows up when one of the radii goes to zero (respectively to infinity), existence of a minimizer in this class is easily obtained.
Let us prove that a minimizer must be either a single ball or a single annulus. If one of the connected components of $E$ is a ball then $E$ must be the ball. Indeed, by the first part of the proof, it cannot contain two balls. Moreover, the union of a ball and an annulus 
has energy higher than that of a ball of area the sum of the areas. Indeed, by the isoperimetric inequality the perimeter is better for the ball and since the elastica energy of a ball is a decreasing function
of its area, the elastica energy of a single ball is also better than the elastica energy of a ball and an annulus. 

We are left to prove that one annulus is better than two. For $i=1,2$, let $r_i$ be the internal radii and $R_i$ be the external radii so that $m_i= \pi(R_i^2-r_i^2)$ are the area of the two annuli.
We consider as a competitor the annulus $B_R\backslash B_{r_1}$ with 
\begin{equation*}
  R^2= R_1^2+(R_2^2-r_2^2).
\end{equation*}
Since the elastica part of the energy is smaller, we are left to prove that the perimeter part is smaller, too. That is indeed the case since 
\begin{equation*}
   R=(R_1^2+R_2^2-r_2^2)^{1/2}\le R_1+R_2+r_2.
\end{equation*}
\item
Let us now prove the existence of the threshold $\bar \lambda$. Let us start by noticing that if the ball is a minimizer of \eqref{probisoper} for some $\lambda$ then it is also a minimizer of \eqref{probisoper} for every $\lambda'> \lambda$. Indeed, by the isoperimetric inequality, for every set $E\in\M(|B_1|)$,
\begin{equation*}
   \F_{\lambda'}(E)=\F_\lambda(E)+(\lambda'-\lambda)P(E)\ge \F_\lambda(B_1)+(\lambda'-\lambda)P(B_1)=\F_{\lambda'}(B_1)
\end{equation*}
with equality if and only if $E$ is a ball. 

The energy of an annulus of internal radius $r$ and area equal to $|B_1|$ is 
larger than the energy of the ball $B_1$ if and only if
\begin{equation}\label{ballwin}
 2\pi \lt(\lambda +1\rt)< \min_{r>0} f_\lambda(r), 
\end{equation}
for $f_\lambda$ as defined in \eqref{eq:flambda}.
By taking as competitor $r=\lambda^{-1/2}$, we get that 
\begin{equation*}
  \min_r f_\lambda(r)\le 2\pi \lambda^{1/2}\lt(2+(1+\lambda)^{1/2}+(1+\lambda)^{-1/2}\rt),
\end{equation*}
so that for all $\lambda>0$ sufficiently small  \eqref{ballwin} does not hold and minimizers are annuli. 

Using that $\lambda r+ r^{-1}\ge 2 {\lambda}^{1/2}$ and   $(1+r^2)^{1/2}\ge 1$, we obtain the lower bound 
\begin{equation*}
   \min_r f_\lambda(r)\ge 2\pi (2\lambda^{1/2}+\lambda),
\end{equation*}
proving that \eqref{ballwin} holds for $\lambda \ge \frac{\sqrt{2}}{2}$ and concluding the proof.
\end{enumerate}
\end{proof}

\begin{remark}
Finding the explicit value of $\bar \lambda$ is not straightforward. Indeed, this entails first minimizing $f_\lambda(r)$ and then finding the range of $\lambda$ such that \eqref{ballwin} holds. The unique minimizer $r_\lambda
$ of $f_\lambda$ is determined by \eqref{ELannulus}. 

By \eqref{ballwin} we deduce that the ball is the minimizer if and only if
\begin{equation*}
 g(\lambda)= 2\pi (\lambda +1) - f_\lambda(r_\lambda
) < 0.
\end{equation*}
By minimality of $ f_\lambda$ in $r_\lambda
$ we find
\begin{equation*}
	g'(\lambda)\,=\, 2\pi - 2\pi\big(r_\lambda
+(1+r_\lambda
^2)^{\frac{1}{2}}\big)\,<\, 0.
\end{equation*}
Therefore, the threshold $\bar\lambda$ is characterized by the condition
\begin{equation*}
	g(\bar \lambda)\,=\,0.
\end{equation*}
Letting $U=(1+r^{-2})^{1/2}$, we observe that finding  $r_\lambda
$ is equivalent to find   the unique solution $U\ge 1$ of
\begin{equation*}
  \lambda U^2(1+U)=(U^2-1)(1+U^3).
\end{equation*}
Dividing by $1+U$, we are left with finding a root larger than one of the fourth order polynomial
\begin{equation*}
   U^4-U^3-\lambda U^2+U-1.
\end{equation*}
Nevertheless, a simple explicit formula for $\bar\lambda$ seems not to be available.
\end{remark}

\subsection{Stability estimates}
We now turn to stability estimates for the functional $\F_{\lambda}$. As for the proof of Theorem \ref{thm:quantWill}, we will first need to know that sets with small energy are close (in a non-quantitative way) to minimizers. 
\begin{lemma}\label{lem:conv-Alambda}
Consider a sequence of positive numbers $(\lambda_n)_n$  with $\lambda_n\to \lambda \in (0,\infty]$ and a sequence $(E_n)_n$ in $\calM(|B_1|)$ such that
\begin{equation}\label{convergmin}
	\F_{\lambda_n}(E_n)-\min_{\calM(|B_1|)}\F_{\lambda_n} \to 0\quad\text{ as }n\to\infty. 
\end{equation}
Then, up to translations and passing to a subsequence, $E_n$ converges strongly in $W^{2,2}$ to a minimizer of $\F_\lambda$ in $\M(|B_1|)$.

In particular, for all $n\in\N$ sufficiently large $E_n$ is connected and has topological genus one if $\lambda<\bar \lambda$,  zero if $\lambda>\bar \lambda$ and   zero or one if $\lambda=\bar\lambda$.
\end{lemma}
\begin{proof}
We may assume without loss of generality that $2\lambda\ge \lambda_n\ge \lambda/2$ for all $n\in\N$.

Let us first prove that for $n$ large enough, $E_n$ must be connected. For the sake of contradiction, assume it is not. Arguing as in the proof of Theorem \ref{theo:minFlambda}, 
we see that we can  replace  $E_n$ by a set $\widetilde E_n$ made of two connected components each of which is either a ball or an annulus with $\F_{\lambda_n}(\widetilde{E}_n)\le \F_{\lambda_n}(E_n)$ so that \eqref{convergmin} still holds for $\widetilde{E}_n$.
Inspecting the proof of Theorem \ref{theo:minFlambda} we see that we reach a contradiction since the minimum of $\F_{\lambda_n}$ in this class is larger than $\min_{\calM(|B_1|)}\F_{\lambda_n}$ by a constant not depending on $n$. 
If $\lambda<\bar \lambda$ we further obtain that for $n$ large enough, $E_n$ cannot be simply connected since otherwise we would have  by Theorem \ref{theo:minFlambda} for some $\delta(\lambda)$
\[\F_{\lambda_n}(E_n)\ge \F_{\lambda_n}(B_1)\ge  \min_{\calM(|B_1|)}\F_{\lambda_n} +\delta,\]
contradicting again \eqref{convergmin}.
 Let us prove that for any $\lambda$, $E_n$ is of genus at most one. Otherwise, we can write $E_n=F_n\backslash \cup_{i=1}^{N_n} G_i$ with $F_n$ and $G_i$ simply connected and $N_n\ge 2$. 
Letting $V_n=|F_n|$ and $m_i^n=|G_i|$, we have  $V_n-\sum_i m_i^n=|B_1|$ and by \cite{BucurHenrot,FeKN16} and the isoperimetric inequality, that 
\begin{equation*}
  \F_{\lambda_n}(E_n)\ge 2\pi^{3/2} \lt( V_n^{-1/2} +\sum_i (m_i^n)^{-1/2}\rt)+ 2\sqrt{\pi} \lambda_n \lt(V_n^{1/2} +\sum_i (m_i^n)^{1/2}\rt).
\end{equation*}
Since $g_n(x)= \pi^{3/2} x^{-1/2} +\sqrt{\pi} \lambda_n x^{1/2}$ is subadditive, and $N_n\ge 2$,  
\begin{equation*}
  \F_{\lambda_n}(E_n)\ge 2\min_{V-(m_1+m_2)=|B_1|} g_n(V)+g_n(m_1)+g_n(m_2).
\end{equation*}
Notice that $g_n\to \infty$ as $x$ tends to zero or infinity so that the minimum on the right-hand side is attained for $m_1$ and $m_2$ uniformly bounded above and below by a constant depending only on $\lambda$. Since for such values of $m_i$
\[
 g_n(m_1)+g_n(m_2)\ge g_n(m_1+m_2) +\delta(\lambda),
\]
for some $\delta(\lambda)>0$, we have by the last two inequalities and Theorem \ref{theo:minFlambda}
\[
 \F_{\lambda_n}(E_n)\ge 2\min_{V-m=|B_1|} g_n(V)+g_n(m) +\delta(\lambda)\geq \min_{\calM(|B_1|)}\F_{\lambda_n} +\delta(\lambda),
\]
contradicting \eqref{convergmin}.
  Thus, for $n$ large enough,  $E_n=F_n\backslash G_n$ with $F_n$ and $G_n$ simply connected (where $G_n=\emptyset$ is possible if and only if $\lambda\ge\bar\lambda$).

Consider first the case $G_n\not=\emptyset$ for all $n$ sufficiently large  (possibly up to a subsequence). This implies in particular that $\lambda \le \bar \lambda$. Let $A_n=B_{R_1^n}\setminus B_{R_2^n}$ be the centered annulus minimizing $\F_{\lambda_n}$ in $\calM(|B_1|)$, hence $R_2^n=r_{\lambda_n}$, $R_1^n=\sqrt{1+r_{\lambda_n}^2}$, where 
$r_{\lambda_n}$ is the unique minimum of $f_{\lambda_n}$ given in Lemma \ref{rem:annuli}. Choosing $\widetilde{R}_1^n$, $\widetilde{R}_2^n$ with $|B_{\widetilde{R}_1^n}|=|F_n|$, $|B_{\widetilde{R}_2^n}|=|G_n|$ we deduce 
\begin{align*}
	\F_{\lambda_n}(E_n)- \F_{\lambda_n}(A_n)
	&=\, \Big(\lambda_n P(F_n)+W(F_n)+\lambda_n P(G_n)+W(G_n)\Big)\\
	&\qquad -\Big[\lambda_n P(B_{{R}_1^n})+W(B_{{R}_1^n})+\lambda_n P(B_{{R}_2^n})+W(B_{{R}_2^n})\Big]\\
	&=\, \Big(\ldots\Big) - \Big[\lambda_n P(B_{\widetilde{R}_1^n})+W(B_{\widetilde{R}_1^n})+\lambda_n P(B_{\widetilde{R}_2^n})+W(B_{\widetilde{R}_2^n})\Big]\\
	&\qquad +f_{\lambda_n}(\widetilde{R}_2^n)-f_{\lambda_n}({R}_2^n)\\
	&\geq\, \lambda_n\Big(P(F_n)-P(B_{\widetilde{R}_1^n}) + P(G_n)-P(B_{\widetilde{R}_2^n}) \Big)\\
	&\qquad + \Big(W(F_n) - W(B_{\widetilde{R}_1^n}) + W(G_n) - W(B_{\widetilde{R}_2^n})\Big)\\
	&\qquad +\Big(f_{\lambda_n}(\widetilde{R}_2^n)-f_{\lambda_n}({R}_2^n) \Big)
	\qquad=\, T_1^n + T_2^n + T_3^n.
\end{align*}
By \eqref{convergmin} the left-hand side of the inequality vanishes as $n\to\infty$. Since $F_n$ and $G_n$ are simply connected, all three terms on the right-hand side are non-negative and must therefore converge to zero. 
Since $2\lambda\ge \lambda_n\geq\lambda/2$ we  have that $R_2^n$ is uniformly bounded from above and below. Since $T_3^n\to 0$ and  $R_2^n$ minimizes $f_{\lambda_n}$ we deduce by strict convexity of $f_{\lambda_n}$ and a Taylor expansion that
\begin{equation*}
	0\,=\, \lim_{n\to\infty} T_3^n\,\geq\, c(\lambda)\limsup_{n\to\infty} \, (\widetilde{R}_2^n-{R}_2^n)^2, 
\end{equation*}
and by the mass constraint, also $\widetilde{R}_1^n-{R}_1^n\to 0$. 
Since $T_2^n\to 0$ we deduce that $W(F_n)- W(B_{R_1^n})\to 0$ and $W(G_n)- W(B_{R_2^n})\to 0$. After  rescaling we can apply Lemma \ref{lem:quantWill2} and the conclusion follows.

Let us finally consider the case that $\lambda\ge\bar\lambda$  with $E_n$ simply connected. Then,
\begin{equation*}
	 \F_{\lambda_n}(E_n)- \min_{\M(|B_1|)}\F_{\lambda_n}\,
	=\, \lambda_n \big( P(E_n) -P(B_1)\big) +W(E_n)-W(B_1) + \big(\F_{\lambda_n}(B_1) - \min_{\M(|B_1|)}\F_{\lambda_n}\big),
\end{equation*}
and we obtain as in the previous case  that $W(E_n)\to W(B_1)$  concluding again by Lemma~\ref{lem:quantWill2}.
\end{proof}
We can now prove our stability estimate for $\F_\lambda$.

\begin{theorem}\label{thm:stab}
 Let $\bar \lambda$ be given by Theorem \ref{theo:minFlambda} and consider an arbitrary set $E\in \M(|B_1|)$. Then, there exists a universal constant $c_2>0$, such that for $\lambda>\bar \lambda$
\begin{equation}\label{quantWlambdabig}
  \F_{\lambda}(E)-\F_{\lambda}(B_1)\ge c_2 (\lambda-\bar \lambda)\min_x |E\Delta B_1(x)|^2,
\end{equation}
while for any  $\lambda_*>0$ there exists a constant $c(\lambda_*)>0$ such that for any $\lambda\in [\lambda_*,\bar\lambda]$
\begin{equation}\label{quantWlambdasmall}
 \F_{\lambda}(E)- \min_{\M(|B_1|)}\F_{\lambda}\ge c(\lambda_*) \min_{\Omega} |E\Delta \Omega|^2,
\end{equation}
where the minimum is taken among all sets 
$\Omega\in \M(|B_1|)$ such that $\Omega$ is a ball or $\Omega$
is an annulus minimizing $\F_{\lambda}$ in $\M(|B_1|)$.
\end{theorem}

\begin{proof}
 Let us start by proving \eqref{quantWlambdabig}. Since $\lambda-\bar \lambda >0$,  using the minimality of $B$ for $\F_{\bar \lambda}$ and the quantitative isoperimetric inequality \cite{FuMaPra}, 
 we can write that 
 \begin{align*}
  \F_{\lambda}(E)-\F_{\lambda}(B_1)&=\F_{\bar \lambda} (E)-\F_{\bar \lambda} (B_1)+(\lambda-\bar \lambda) \big(P(E)-P(B_1)\big)\\
  &\ge (\lambda-\bar \lambda) \big(P(E)-P(B_1)\big)\\
  &\geq  c_2(\lambda-\bar \lambda) \min_x |E\Delta B_1(x)|^2,
 \end{align*}
which proves \eqref{quantWlambdabig}.\\

We now turn to the proof of \eqref{quantWlambdasmall}. As for the proof of \eqref{quantWill}, we assume by contradiction that the inequality does not hold. We thus have sequences $(\lambda_n)_n$ in $[\lambda_*,\bar\lambda]$ and $(E_n)_n$ in $\M(|B_1|)$ with 
\begin{equation}\label{hypabsurdFlambda}
 \lim_{n\to \infty} \frac{\F_{\lambda_n}(E_n)-\min_{\M(|B_1|)}\F_{\lambda_n}}{\min_\Omega |E_n\Delta \Omega|^2}=0.
\end{equation}
Again, implicit constants in $\les$ and $\ges$ estimates may depend on the fixed sequence $(E_n)_n$ but are independent of $n\in\N$.\\
Since the denominator in \eqref{hypabsurdFlambda} is uniformly bounded we have $\lim_{n\to \infty} \F_{\lambda_n}(E_n)-\min_{\M(|B_1|)}\F_{\lambda_n}=0$. Without loss of generality we can also assume $\lambda_n\to\lambda \in [\lambda_*,\bar\lambda]$ as $n\to\infty$. 

By Lemma \ref{lem:conv-Alambda} $E_n$ converges strongly in $W^{2,2}$ to a minimizer of $\F_\lambda$. For $n$ large enough, it is connected, and has either genus one or zero (the latter being possible only if $\lambda=\bar\lambda$). 

 Let us first consider the case where up to a translation and up to passing to a subsequence, $E_n$ converges to $B_{R_1}\backslash B_{R_2}(x_2)$ so that for $n$ large enough 
$E_n=F_n\setminus G_n$ for some simply connected sets $G_n\subset F_n$. Up to a translation, we may assume that $\dashint_{F_n} x=0$ so that by Lemma  \ref{lem:conv-Alambda}, we can write $F_n$ as a graph over the ball of radius $\widetilde R_1^n$ with $|B_{\widetilde R_1^n}|=|F_n|$,  and such that $\widetilde R_1^n\to R_1$.
At the same time $G_n$ can be written as a graph over $B_{\widetilde R_2^n}(x_n)$ with $|B_{\widetilde R_2^n}|=|G_n|$, $\dashint_{G_n} x=x_n$,  $\widetilde R_2^n \to R_2$ and $x_n\to x_2$ as $n\to \infty$. Hence,
\begin{equation*}
   \partial F_n=\{ \widetilde R_1^n (1+\phi_n)e^{i\theta}, \ \theta\in [0,2\pi)\}, \quad \partial G_n=\{ x_n+ \widetilde R_2^n(1+\psi_n)e^{i\theta}, \ \theta\in[0,2\pi)\}
\end{equation*}
for some functions $\phi_n$, $\psi_n$ with small $W^{2,2}$ norm that satisfy \eqref{volumeconstraint} and \eqref{barycenter}. Let us point out that $B_{\widetilde{R}_1^n}\backslash B_{\widetilde{R}_2^n}$ is in general not an optimal annulus
and is thus in particular not admissible for the denominator of \eqref{hypabsurdFlambda}.\\
By \eqref{quantWillspher} and the Sobolev inequality, we obtain that
\begin{align}
\F_{\lambda_n}(E_n)-\F_{\lambda_n}(B_{\widetilde R_1^n}\backslash B_{\widetilde R_2^n})&\geq 
	W(F_n)-W(B_{\widetilde R_1^n}) + W(G_n)-W(B_{\widetilde R_2^n}) \notag\\
	&\ges  \sup |\phi_n|^2+ \sup|\psi_n|^2+ |F_n\Delta B_{\widetilde R_1^n}|^2+ |G_n\Delta B_{\widetilde R_2^n}(x_n)|^2.
	\label{eq:estimFlambda1}
\end{align}
If $B_{\widetilde{R}_2^n}(x_n)\not\subset B_{\widetilde R_1^n}$, we   need to move $x_n$ inwards to obtain an annulus. To this aim, let 
\begin{equation}\label{defdeltan}
  \delta_n=\max(|x_n|+\widetilde R_2^n-\widetilde R_1^n,0)
\end{equation}
and define $\widetilde{x}_n= x_n- \delta_n \frac{x_n}{|x_n|}$ 
(notice that if $\delta_n>0$ then $x_n\neq 0$ so that taking as a convention that $\widetilde{x}_n=0$ if $x_n=0$, this quantity is well defined). We then have $B_{\widetilde R_2^n}(\widetilde{x}_n)\subset B_{\widetilde R_1^n}$ by \eqref{defdeltan}. Since 
\[B_{\widetilde{R}_2^n(1-\sup |\psi_n|)}(x_n)\subset G_n\subset F_n\subset B_{\widetilde{R}^n_1(1+\sup |\phi_n|)},\] we must have $|x_n|+\widetilde{R}_2^n(1-\sup |\psi_n|)\le \widetilde{R}^n_1(1+\sup |\phi_n|)$ that is
\begin{equation}\label{bounddelta}
 \delta_n^2\le \lt(\widetilde R_1^n \sup |\phi_n|+\widetilde R_2^n\sup |\psi_n|\rt)^2\stackrel{\eqref{eq:estimFlambda1}}{\les} \F_{\lambda_n}(E_n)-\F_{\lambda_n}(B_{\widetilde R_1^n}\backslash B_{\widetilde R_2^n}).
\end{equation}
From this we deduce that 
\begin{equation}\label{estimdeltaV}
 |G_n\Delta B_{\widetilde R_2^n}(\widetilde{x}_n)|^2\les |G_n\Delta B_{\widetilde R_2^n}(x_n)|^2+\delta_n^2 \stackrel{\eqref{eq:estimFlambda1}, \eqref{bounddelta}	}{\les}\F_{\lambda_n}(E_n)-\F_{\lambda_n}(B_{\widetilde R_1^n}\backslash B_{\widetilde R_2^n}).
\end{equation}
We can now estimate
\begin{align*}
 \F_{\lambda_n}(E_n)-\F_{\lambda_n}(A_n)
 & =\F_{\lambda_n}(E_n)-\F_{\lambda_n}(B_{\widetilde R_1^n}\backslash B_{\widetilde R_2^n})+\F_{\lambda_n}(B_{\widetilde R_1^n}\backslash B_{\widetilde R_2^n})-\F_{\lambda_n}(A_n)\\
 & \stackrel{\mathmakebox[\widthof{=}]{\eqref{eq:estimFlambda1},\eqref{estimdeltaV}}}{\ges} 
 |F_n\Delta B_{\widetilde{R}^n_1}|^2+|G_n\Delta B_{\widetilde R_2^n(\widetilde{x}_n)}|^2+\F_{\lambda_n}(B_{\widetilde R_1^n}\backslash B_{\widetilde R_2^n})-\F_{\lambda_n}(A_n)\\
 &\ges |E_n\Delta (B_{\widetilde R_1^n}\backslash B_{\widetilde R_2^n}(\widetilde{x}_n))|^2 +\F_{\lambda_n}(B_{\widetilde R_1^n}\backslash B_{\widetilde R_2^n})-\F_{\lambda_n}(A_n).
\end{align*}
By the minimality of $f_{\lambda_n}$ 
at $r=R_2^n$ and by strict convexity of $f_{\lambda_n}$ (see Lemma \ref{rem:annuli}), $f_{\lambda_n}''(R_2^n)>0$  and thus
\begin{equation*}
   \F_{\lambda_n}(B_{\widetilde R_1^n}\backslash B_{\widetilde R_2^n})-\F_{\lambda_n}(A_n)\ges c(\lambda) (\widetilde R_2^n-R_2^n)^2.
\end{equation*}
We therefore conclude that 
\begin{equation}\label{eq:estimFlambda2}
 \F_{\lambda_n}(E_n)-\F_{\lambda_n}(A_n)\ges |E_n \Delta (B_{\widetilde R_1^n}\backslash B_{\widetilde R_2^n}(\widetilde{x}_n))|^2+ |\widetilde R_2^n-R_2^n|^2.
\end{equation}
As pointed out above, this is not sufficient to obtain a contradiction with \eqref{hypabsurdFlambda} since $B_{\widetilde{R}_1^n}\backslash B_{\widetilde{R}_2^n}(\widetilde x_n)$ is not an optimal annulus for $\F_{\lambda_n}$. We thus need to prove
that there exists $\widehat{x}_n$ close to $\widetilde{x}_n$ such that \eqref{eq:estimFlambda2} holds with $B_{R_1^n}\backslash B_{R_2^n}(\widehat x_n)$ instead of $B_{\widetilde R_1^n}\backslash B_{\widetilde R_2^n}(\widetilde{x}_n)$.

For $\eps_n\ge0$, we let $\widehat{x}_n=(1-\eps_n)\widetilde{x}_n$. We want to choose $\eps_n$ so that $B_{R_2^n}(\widehat{x}_n)\subset B_{R_1^n}$. 
If $|\widetilde{x}_n|\ll 1$, we set $\eps_n=0$ while for $|\widetilde{x}_n|\ges 1$, we claim that we can take 
\begin{equation}\label{eq:estimepsn}
 \eps_n=C |\widetilde R_2^n-R_2^n|
\end{equation}
for some constant $C\les 1$.
Indeed, if $\eps_n$ satisfies \eqref{eq:estimepsn}, then using first that by \eqref{defdeltan} $|\widetilde{x}_n|+\widetilde R_2^n\le \widetilde R_1^n$ and then that  $|\widetilde R_1^n-R_1^n|\sim |\widetilde R_2^n-R_2^n|$ 
(since $B_{\widetilde R_1^n}\setminus B_{\widetilde R_2^n}$ and $B_{R_1^n}\setminus B_{R_2^n}$ have equal mass) we obtain
\begin{align*}
 |\widehat{x}_n|+ R_2^n&= (1-\eps_n)|\widetilde{x}_n|+ \widetilde R_2^n + (R_2^n-\widetilde R_2^n)\\
 &\le \widetilde R_1^n-(C|\widetilde{x}_n|-1) |\widetilde R_2^n-R_2^n|\\
 &=R_1^n + (\widetilde R_1^n-R_1^n)-(C|\widetilde{x}_n|-1) |\widetilde R_2^n-R_2^n|\\
 &\le R_1^n -(C|\widetilde{x}_n|-c)|\widetilde R_2^n-R_2^n|\\
 &\stackrel{|\widetilde{x}_n|\ges 1}{\le} R_1^n,
\end{align*}
 and therefore $B_{R_2^n}(\widehat{x}_n)\subset B_{R_1^n}$. From \eqref{eq:estimepsn} we conclude that 
\begin{equation*}
   |(B_{\widetilde R_1^n}\Delta B_{\widetilde R_2^n}(\widetilde{x}_n))\Delta(B_{R_1^n}\backslash B_{R_2^n}(\widehat{x}_n))|^2 \les |\widetilde R_1^n-R_1^n|^2+|\widetilde R_2^n-R_2^n|^2+|\widetilde{x}_n-\widehat{x}_n|^2\les |\widetilde R_2^n-R_2^n|^2
\end{equation*}
so that \eqref{eq:estimFlambda2} and the fact that $B_{R_1^n}\backslash B_{R_2^n}(\widehat{x}_n)$ is an optimal annulus for $\F_{\lambda_n}$, finally yields
\begin{multline*}
 \F_{\lambda_n}(E_n)-\F_{\lambda_n}(A_n)\\
 \ges |E_n \Delta (B_{\widetilde R_1^n}\backslash B_{\widetilde R_2^n}(\widetilde{x}_n))|^2+ |(B_{\widetilde R_1^n}\backslash B_{\widetilde R_2^n}(\widetilde{x}_n))\Delta(B_{R_1^n}\backslash B_{R_2^n}(\widehat{x}_n))|^2\\
 \ges |E_n \Delta (B_{R_1^n}\backslash B_{R_2^n}(\widehat{x}_n))|^2\ge \min_{\Omega} |E_n\Delta \Omega|^2,
\end{multline*}
which contradicts \eqref{hypabsurdFlambda}.

Let us finally consider the case  $\lambda=\bar\lambda$ with $E_n$ simply connected. We then have $\F_{\lambda_n}(E_n)\to \inf_{\M(|B_1|)}\F_{\bar\lambda}=\F_{\bar\lambda}(B_1)$ and $W(E_n)\ge W(B_1) $ by \cite{BucurHenrot,FeKN16} so that 
the quantitative isoperimetric inequality \cite{FuMaPra} gives a contradiction to \eqref{hypabsurdFlambda} (one could also use \eqref{quantWill}).
\end{proof}

%
%
\section{The planar case: charged drops}\label{sec:planarcharged}
Still considering the planar case $d=2$, we now turn our attention to the variational problem \eqref{problem2d-gen} for arbitrary positive parameters $m$, $\lambda$ and $Q$. 
Our aim is to  understand as much as possible the phase diagram of $\F_{\lambda,Q}$ i.e.~identify regions of existence/non-existence of minimizers and characterize minimizers when they exist. 

By a simple rescaling we have
\begin{equation}
	\min_{\M(m)}\F_{\lambda,Q}(E)\,=\,  \frac{\sqrt \pi}{\sqrt m}\min_{\M(|B_1|)}\F_{\lambda(m),Q(m)}(E), \label{eq:rescale-FlambdaQ}
\end{equation}
with $\lambda(m)=\lambda \frac{m}{\pi}$, $Q(m)=Q\left(\frac{m}{\pi}\right)^{\frac{3+\alpha}{2}}$ so that we may assume that $m=|B_1|$. 
Notice however that, for $\lambda$ and $Q$ fixed, $\lambda(m)$ and $Q(m)$ tend to zero as $m$ goes to zero. Therefore, if we want to understand the shape of minimizers at small volume,
we need to carefully study the phase diagram of $\F_{\lambda,Q}$ close to $(\lambda,Q)=(0,0)$.

\subsection{Minimization in the class of simply connected sets}
Let us start by investigating \eqref{problem2d}, i.e.~restricting ourselves to simply connected sets.
\begin{proposition}\label{prop:simply_connected}
There exists $Q_0>0$ such that for $Q<Q_0$ and all $\lambda\geq 0$, balls are the only solutions of  the minimization problem
\begin{equation*}
	\min_{\M_{sc}(|B_1|)}\F_{\lambda,Q}(E).
\end{equation*} 
\end{proposition}
\begin{proof}
Consider $\delta_0,c_1>0$ from Theorem \ref{thm:quantWill}. By \cite[Proposition 7.1]{KnMu13} (see also \cite[Theorem 1.3]{F2M3}) there exists $Q_1$ such that
\begin{equation}\label{eq:KnMu13}
	P(E)-P(B_1) \geq Q_1\big(V_\alpha(B_1)-V_\alpha(E)\big)\quad\text{ for all }E\in \M(|B_1|).
\end{equation}
Consider now any $E\in \M_{sc}(|B_1|)$ with $\F_{\lambda,Q}(E)\leq\F_{\lambda,Q}(B_1)$
Hence 
\begin{align}\label{estimW}
	W(E)-W(B_1) &\leq  Q \big(V_\alpha(B_1)-V_\alpha(E)\big) -\lambda \big(P(E)-P(B_1)\big)\\
	&\leq QV_\alpha(B_1)\,<\,\delta_0 \notag
\end{align}
for all $Q< \delta_0 V_\alpha(B_1)^{-1}$.Therefore, if $Q<\delta_0 V_\alpha(B_1)^{-1}$, \eqref{quantWill2} applies and 
\[
 c_1\frac{P(E)-P(B_1)}{P(B_1)}\le W(E)-W(B_1).
\]
Together with \eqref{estimW} we get that for all $Q<\delta_0 V_\alpha(B_1)^{-1}$
\[
 P(E) - P(B_1) \le  Q \Big(\frac{c_1}{P(B_1)}+\lambda\Big)^{-1}\big(  V_\alpha(E)- V_\alpha(B_1)\big),
\]
which combined with \eqref{eq:KnMu13} gives $P(E)=P(B_1)$ for $Q<Q_1 (\frac{c_1}{P(B_1)}+\lambda)$. This implies that $E=B_1$. Choosing $Q_0:=\min\{\delta_0 V_\alpha(B_1)^{-1},Q_1 \frac{c_1}{P(B_1)}\}$ concludes the proof of the proposition. 
\end{proof}

\begin{remark}
 Since for $d=2$ a bound on the perimeter gives a bound on the diameter for simply connected sets, existence of a minimizer for \eqref{problem2d} holds 
 for every $\lambda>0$ and $Q>0$. For $\lambda=0$, the existence of minimizers for large $Q$ is less clear. 
\end{remark}
As a direct consequence of  \eqref{eq:rescale-FlambdaQ} and Proposition \ref{prop:simply_connected} we get
\begin{corollary}\label{cor:mass-rescale}
Let $Q_0$ be given by Proposition \ref{prop:simply_connected}. For any $Q>0$, $\lambda\geq 0$ and  $m\le\pi\lt(\frac{Q_0}{Q}\rt)^{\frac{2}{3+\alpha}}$  balls are the only minimizers of
\begin{equation*}
 \min_{\calM_{sc}(m)} \F_{\lambda,Q}(E).
\end{equation*}
\end{corollary}

\subsection{Minimization in the class $\M(|B_1|)$.}

We now drop the constraint that $E$ is simply connected and study \eqref{problem2d-gen}. We start by focusing on the simplest part of the phase diagram, that is where
minimizers are balls. As above, let $Q_1>0$ be given by \cite{KnMu13,F2M3} such that balls are the only minimizers of 
\[
 \min_{\M(|B_1|)} P(E)+Q V_\alpha(E).
\]

\begin{proposition}[Global minimality of the ball for $\lambda> \bar \lambda$]
\label{prop:stabball}
 For every $\lambda>\bar \lambda$ and $Q\le Q_1(\lambda-\bar \lambda)$  balls are the unique minimizers of $\F_{\lambda, Q}$ in $\calM(|B_1|)$.
\end{proposition}
\begin{proof}
For $\lambda>\bar \lambda$ and $Q\le Q_1(\lambda-\bar \lambda)$, we have
\[
 \F_{\lambda,Q}(E) = (\lambda-\bar \lambda) P(E)+Q V_\alpha(E) +\F_{\bar \lambda}(E)=  (\lambda-\bar \lambda) \lt(P(E)+ \frac{Q}{\lambda-\bar \lambda}V_\alpha(E)\rt) +\F_{\bar \lambda}(E).
\]
By definition of $Q_1$, balls are the only minimizers of $P(E)+ \frac{Q}{\lambda-\bar \lambda}V_\alpha(E)$. Since by Theorem~\ref{theo:minFlambda} they also minimize $\F_{\bar \lambda}$, 
balls are the only minimizers of $\F_{\lambda,Q}$. 
\end{proof}
\begin{remark}
 Notice that if balls are minimizers of $\F_{\lambda,Q}$ then by the isoperimetric inequality they are also minimizers  of $\F_{\lambda',Q}$ for every $\lambda'\ge \lambda$.
\end{remark}

We now focus on the most interesting case and show that for $\lambda\le\bar\lambda$ and $Q$ sufficiently small centered annuli are optimal.
Our first observation is that among annuli of the form $B_{R_1}\backslash B_{R_2}(x_2)$, the Riesz interaction energy is minimized for the centered annulus i.e.~for $x_2=0$.
\begin{lemma}\label{lem:RieszA}
 For every $x_2\in\R^2$ such that $B_{R_2}(x_2)\subset B_{R_1}$, it holds
 \begin{equation*}
  V_\alpha(B_{R_1}\backslash B_{R_2}(x_2))\ge V_\alpha(B_{R_1}\backslash B_{R_2}),
 \end{equation*}
with equality if and only if $x_2=0$. Moreover $V_\alpha(B_{R_1}\backslash B_{R_2}(x_2))\le V_\alpha(B_1)$.
\end{lemma}
\begin{proof}
Let $\Chi_1=\Chi_{B_{R_1}}$, $\Chi_2=\Chi_{B_{R_2}}$.  By  the Riesz rearrangement inequality we deduce
\begin{align*}
	V_\alpha(B_{R_1}\backslash B_{R_2}(x_2)) \,&=\, \int_{\R^2\times \R^2} \frac{\Chi_1(x)\Chi_1(y)}{|x-y|^{2-\alpha}} + \int_{\R^2\times \R^2} \frac{\Chi_2(x-x_2)\Chi_2(y-x_2)}{|x-y|^{2-\alpha}}\\
	&\qquad\qquad\qquad -2 \int_{\R^2\times \R^2} \frac{\Chi_1(x)\Chi_2(y-x_2)}{|x-y|^{2-\alpha}}\\
	&\geq\, \int_{\R^2\times \R^2} \frac{\Chi_1(x)\Chi_1(y)}{|x-y|^{2-\alpha}} + \int_{\R^2\times \R^2}\frac{ \Chi_2(x)\Chi_2(y)}{|x-y|^{2-\alpha}} \\
	&\qquad\qquad\qquad -2 \int_{\R^2\times \R^2} \frac{\Chi_1(x)\Chi_2(y)}{|x-y|^{2-\alpha}}\\
	&=\, V_\alpha(B_{R_1}\backslash B_{R_2}).
\end{align*}
By \cite[Theorem 3.9]{LiLo01} equality holds if and only if $\Chi_1,\Chi_2$ coincide with their symmetric rearrangement, hence if and only if $x_2=0$. 
The last statement is a direct consequence of the Riesz rearrangement inequality. 
\end{proof}

We will also need the following stability lemma.
\begin{lemma}\label{lem:EQ_to_Alambda}
For any $\lambda_*\in(0,\bar \lambda]$ there exists $C(\lambda_*)>0$ such that for all $\lambda\in [\lambda_*,\bar\lambda]$ and all $Q>0$
the following property holds: Let $E\in\M(|B_1|)$ satisfy
\begin{equation}\label{hypEQ}
	\F_{\lambda,Q}(E)\le \F_{\lambda,Q}(A_\lambda),
\end{equation}
where $A_\lambda$ is the centered annulus that minimizes $\F_\lambda$ in $\M(|B_1|)$.\\ Then, there exists $\Omega\in\M(|B_1|)$ 
which is either a ball or an annulus minimizing $\F_\lambda$, such that
\begin{equation}
	 |E\Delta \Omega| + \big( V_\alpha(\Omega)-V_\alpha( A_\lambda)\big)   \,\leq\, C(\lambda_*)Q. \label{eq:lem3.6}
 \end{equation}
 Moreover, for $Q$ small enough $\Omega$ must be an annulus minimizing $\F_\lambda$. 
\end{lemma}
\begin{proof}
Let $\Omega\in\argmin_{\M(|B_1|)} \F_\lambda \cup \{B_1(x):x\in\R^2\}$ satisfy 
\begin{equation*}
	|E\Delta \Omega| \,\leq\,  |E\Delta \Omega'| \quad\text{ for all }\Omega'\in \argmin_{\M(|B_1|)} \F_\lambda\cup \{B_1(x):x\in\R^2\}.
\end{equation*}
Thanks to \eqref{hypEQ}, \eqref{quantWlambdasmall}, and the Lipschitz-continuity of the Riesz interaction energy 
\begin{align*}
   c(\lambda_*)  |E\Delta \Omega|^2 &\leq \F_{\lambda}(E)- \F_{\lambda}(A_\lambda)\\
   &\leq Q\big(V_\alpha(A_\lambda)-V_\alpha(\Omega)\big)+Q\big(V_\alpha(\Omega)-V_\alpha(E)\big)\\
   &\leq Q\big(V_\alpha(A_\lambda)-V_\alpha(\Omega)\big)+C Q |E\Delta \Omega|.
 \end{align*}
By Lemma \ref{lem:RieszA} the first term on the right-hand side is non-positive and we deduce \eqref{eq:lem3.6}.\\
Since $V_\alpha(B_1)-V_\alpha(A_\lambda)\ge c(\lambda_*)$, we also conclude from \eqref{eq:lem3.6} that for $Q$ small enough, $\Omega$ cannot be a ball.
\end{proof}
We can now prove the minimality of the centered annulus in this regime.
\begin{theorem}[Global minimality of the annulus for $\lambda\leq \bar \lambda$]
\label{theo:stabannulus}
For every $0<\lambda_*<\bar \lambda$ there exists $ Q(\lambda_*)$ such that for all $Q< Q(\lambda_*)$ and all $\lambda\in [\lambda_*,\bar \lambda]$, the  minimizers of $\F_{\lambda,Q}$ in $\calM(|B_1|)$ are 
centered annuli $A_{\lambda,Q}$.
Moreover, there exist positive constants $c(\lambda_*), C(\lambda_*)$  depending only on $\lambda_*$ such that the inner radius $r_{\lambda,Q}$ of $A_{\lambda,Q}$ satisfies
\begin{equation}\label{estimdeltar}
  c(\lambda_*)Q \le |r_{\lambda,Q}-r_\lambda|\le C(\lambda_*)Q,
\end{equation}
where $r_\lambda$ is the minimizer of $f_\lambda$ (see Lemma \ref{rem:annuli}).
\end{theorem}

\begin{proof}
\begin{enumerate}[align=left, leftmargin=0pt, 
listparindent=\parindent, labelwidth=0pt, itemindent=!,label=\underline{Step \arabic*}.]
\item
We  minimize first $\F_{\lambda,Q}$ in the class of annuli. By \eqref{lem:RieszA}, this minimum is attained by centered annuli i.e.~annuli of the form
 $A_r=B_{R}\backslash B_r$ with $R=R(r)=\sqrt{1+r^2}$. We recall that $f_\lambda(r)=\F_\lambda(A_r)$ and let $g(r)=V_\alpha(A_r)$. Hence, we are left with minimizing
\begin{equation*}
  h_{\lambda,Q}(r)=f_\lambda(r)+Q g(r).
\end{equation*}
Since $f_\lambda$ is coercive and $g$ is positive (and since both are continuous), there exists at least one minimum $r_{\lambda,Q}$ of $h_{\lambda,Q}$. 
\item
We claim that $g'(r)<0$ for all $r>0$. To prove this, let $v(x)=\int_{B_R\setminus B_r} |x-y|^{-2+\alpha}\,dy$ be the potential created by the annulus $B_R\backslash B_r$ then
\begin{align*}
  g'(r) &= \frac{d}{dr} \int_{B_{R(r)}\setminus B_r}\int_{B_{R(r)}\setminus B_r} |x-y|^{-2+\alpha}\,dy\,dx\\
  &= 2 \Big(R'(r)\int_{\partial B_{R(r)}} v(x)\,d\Ha^1(x)-\int_{\partial B_r} v(x)\,d\Ha^1(x)\Big)\\
  &= 2\int_{\partial B_r} v\lt(\frac{R}{r}x\rt)-v(x)\,d\Ha^1(x). 
\end{align*}

For $x\in \partial B_r$, let $S_x=B_R((1+\frac{R}{r})x)\cap A_r$ (see Figure \ref{fig:Sx}).
 \begin{figure}[h!]\begin{center}
 \resizebox{8.cm}{!}{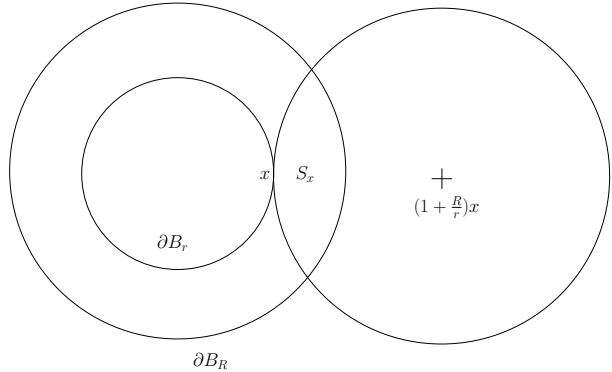}
   \caption{The set $S_x$.} \label{fig:Sx}
 \end{center}
 \end{figure}
By the symmetry of $S_x$ with respect to the line $\lt\{\frac{x}{r}\cdot y=\frac{r+R}{2}\rt\}$, 
\begin{equation*}
   \int_{S_x} \frac{1}{|x-y|^{2-\alpha}}=\int_{S_x}\frac{1}{|\frac{R}{r}x-y|^{2-\alpha}}.
\end{equation*}
Moreover, since for $y\in A_r\cap \bar{S}_x^c$, $|x-y|<|\frac{R}{r}x-y|$, we have for $x\in \partial B_r$
\begin{equation*}
   v\lt(\frac{R}{r}x\rt)-v(x)=\int_{A_r}\frac{1}{|\frac{R}{r}x-y|^{2-\alpha}}-\int_{A_r} \frac{1}{|x-y|^{2-\alpha}}<0
\end{equation*}
so that for every $r>0$, $g'(r)<0$.
\item
We therefore have $r_{\lambda,Q}>r_\lambda$ and that  $r_{\lambda,Q}$ is increasing in $Q$ and decreasing in $\lambda$. Since $ h_{\lambda,Q}(r_\lambda)\geq h_{\lambda,Q}(r_{\lambda,Q})$ and since $f_\lambda$ is strictly convex with $f_\lambda'(r_\lambda)=0$ and $f_\lambda'''<0$ we next deduce
\begin{equation*}
	Q(g(r_\lambda)-g(r_{\lambda,Q})) \geq f_\lambda(r_{\lambda,Q})-f_\lambda(r_\lambda) \geq \frac{1}{2}f_\lambda''(r_{\lambda,Q})(r_{\lambda,Q}-r_\lambda)^2 .
\end{equation*}
On the one hand $g(r_\lambda)\leq V_\alpha(B_1)$ and on the other hand $r_{\lambda,Q}\leq C(\lambda_*)$ and $f_\lambda''(r_{\lambda,Q})\geq c(\lambda_*)>0$ for all $\lambda\geq\lambda_*$, $Q<1$. We  infer that
\begin{equation*}
	(r_{\lambda,Q}-r_\lambda)^2 \le C(\lambda_*)Q.
\end{equation*}
Hence $r_{\lambda,Q}\to r_\lambda$ as $Q\to 0$ uniformly for all $\lambda\in [\lambda_*,\bar \lambda]$.
\item
By the minimizing property of $r_{\lambda,Q}$, we deduce
\begin{equation*}
   0 = f'(r_{\lambda,Q})+Qg'(r_{\lambda,Q}) =  f''(\widetilde r_{\lambda,Q})(r_{\lambda,Q}-r_\lambda) + Qg'(r_{\lambda,Q})
\end{equation*}
for some $r_\lambda\leq \widetilde r_{\lambda,Q}\leq r_{\lambda,Q}$. We have $c(\lambda_*)\leq f''(\widetilde r_{\lambda,Q})\leq C(\lambda_*)$ and since $c(\lambda_*)\leq r_{\lambda,Q}\leq C(\lambda_*)$ also $g'(r_{\lambda,Q})$
is uniformly bounded from above and below from which  \eqref{estimdeltar} follows.
\item
Assume for the sake of contradiction that we can find a sequence $\lambda_n\in[\lambda_*,\bar \lambda]$ converging to $\lambda\in[\lambda_*,\bar \lambda]$ and  sequences $Q_n\to 0$,  $(E_n)_n$ in $\calM(|B_1|)$ which are not annuli and such that
\begin{equation}
	\F_{\lambda_n,Q_n}(E_n)\le \F_{\lambda_n,Q_n}(A_n), \label{eq:hypThm3.7}
\end{equation} 
where $A_n=A_{r_{\lambda_n,Q_n}}$ is the optimal annulus (which is centered). 

By \eqref{eq:hypThm3.7} and the fact that $r_{\lambda_n,Q_n}\to r_{\lambda}$, we see that $E_n$ satisfies \eqref{convergmin} so that Lemma \ref{lem:conv-Alambda} together with Lemma \ref{lem:EQ_to_Alambda} imply that 
 for  $n$ sufficiently large, up to a translation $E_n={F_n}\backslash {G_n}$ with $F_n$ converging strongly in $W^{2,2}$ to $B_{R_\lambda}$ (where $R_\lambda=\sqrt{1+r_\lambda^2}$) and $G_n$ converging strongly in $W^{2,2}$ to $B_{r_\lambda}$.
 Let $R_n\to R_\lambda$ and $r_n\to r_\lambda$ be such that 
 \[
   |B_{R_n}|=|{F_n}|\qquad  \textrm{and} \qquad |B_{r_n}|=|{G_n}|.
 \]
 Since $A_n$ is optimal for $\F_{\lambda_n,Q_n}$ among annuli, we have $\F_{\lambda_n,Q_n}(A_n)\le \F_{\lambda_n,Q_n}(B_{R_n}\backslash B_{r_n})$ so that \eqref{eq:hypThm3.7} becomes
\begin{multline*}
 \lambda_nP(F_n) +W(F_n) +\lambda_n P(G_n)+W(G_n) +Q_n V_\alpha(F_n\backslash G_n)\\
 \le \lambda_nP(B_{R_n}) +W(B_{R_n}) +\lambda_n P(B_{r_n})+W(B_{r_n}) +Q_n V_\alpha(B_{R_n}\backslash B_{r_n}).
\end{multline*}
Using that  by \cite{BucurHenrot,FeKN16}, $W(F_n)\ge W(B_{R_n})$ and $W(G_n)\ge W(B_{r_n})$, this simplifies to
\begin{equation}\label{eq:secstep}
 \lambda_n P(F_n) +\lambda_n P(G_n) +Q_n V_\alpha(F_n\backslash G_n)
 \le \lambda_n P(B_{R_n})  +\lambda_n P(B_{r_n}) +Q_n V_\alpha(B_{R_n}\backslash B_{r_n}).
\end{equation} 
We now estimate (for simplicity we do not write the kernel in the integrals)
\begin{align*}
  &V_\alpha(B_{R_n}\backslash B_{r_n})-V_\alpha({F_n\backslash G_n})\\
  &\qquad = \int_{B_{R_n}}\int_{B_{R_n}}+\int_{B_{r_n}}\int_{B_{r_n}}-2\int_{B_{R_n}}\int_{B_{r_n}}-\int_{G_n}\int_{G_n} - \int_{F_n}\int_{F_n} +2\int_{F_n}\int_{G_n}\\
  &\qquad =V_\alpha(B_{R_n})-V_\alpha(F_n)+V_\alpha(B_{r_n})-V_\alpha(G_n) +2 \lt[\int_{F_n}\int_{G_n}-\int_{B_{R_n}}\int_{B_{r_n}}\rt]
\end{align*}
By the Riesz rearrangement inequality,
\begin{equation*}
   \int_{F_n}\int_{G_n}-\int_{B_{R_n}}\int_{B_{r_n}}\le 0
\end{equation*}
from which we obtain
\[
 V_\alpha(B_{R_n}\backslash B_{r_n})\le V_\alpha({F_n\backslash G_n})+V_\alpha(B_{R_n})-V_\alpha(F_n)+V_\alpha(B_{r_n})-V_\alpha(G_n).
\]
Inserting this into \eqref{eq:secstep} and dividing by $\lambda_n$, we obtain
\[
 P(F_n)+\frac{Q_n}{\lambda_n} V_\alpha(F_n)+  P(G_n)+\frac{Q_n}{\lambda_n} V_\alpha(G_n)\le P(B_{R_n})+\frac{Q_n}{\lambda_n}V_\alpha(B_{R_n})+  P(B_{r_n})+\frac{Q_n}{\lambda_n}V_\alpha(B_{r_n}),
\]
which implies by \cite[Proposition 7.1]{KnMu13} that if $Q_n/\lambda_n$ is small enough then $F_n=B_{R_n}$ and $G_n=B_{r_n}(x_n)$ for some $x_n\in \R^2$. This contradicts our assumption that $E_n$ was not an annulus and concludes the proof.
\end{enumerate}
\end{proof}
Having in mind the study of \eqref{eq:rescale-FlambdaQ} for $\lambda$ and $Q$ fixed but $m$ tending to zero, we now focus on the behavior of $Q(\lambda_*)$ for $\lambda_*$ going to zero.
\begin{proposition}\label{prop:stabannulus2}
There exist $Q_2>0$  and $\lambda_0>0$ such that if $\lambda\in(0, \lambda_0]$ and $Q\le Q_2 \lambda^{\frac{3+\alpha}{2}}$ then every minimizer $E_{\lambda,Q}$ of 
\eqref{eq:rescale-FlambdaQ} is a centered annulus. If $E_{\lambda,Q}=B_{R_{\lambda,Q}}\backslash B_{r_{\lambda,Q}}$ then $\lambda^{1/2}r_{\lambda,Q}\to 1$ and $\lambda^{1/2} R_{\lambda,Q}\to 1$, 
as $\lambda\to 0$.
\end{proposition}

\begin{proof}
Consider sequences $(\lambda_n)_n$, $(Q_n)_n$ with
\begin{equation*}
	\lambda_n\to 0,\qquad Q_n\to 0,\qquad \limsup_{n\to\infty} Q_n\lambda_n^{-\frac{3+\alpha}{2}}\leq Q_2.
\end{equation*}
We start by making the rescaling $E= \lambda_n^{-1/2} \widehat E$ so that 
\begin{equation*}
  \F_{\lambda_n,Q_n}(E)=\lambda_n^{1/2}\lt(P(\widehat E)+W(\widehat E)+ Q_n\lambda_n^{-\frac{3+\alpha}{2}}V_\alpha(\widehat E)\rt)=\lambda_n^{1/2} \F_{1, \widetilde Q_n }(\widehat E),
 \end{equation*}
with $\widetilde Q_n=Q_n \lambda_n^{-\frac{3+\alpha}{2}}$. We are thus left to study the minimization problem
\begin{equation*}
	\min_{\M(\lambda_n\pi)} \F_{1, \widetilde Q_n}( E).
\end{equation*}
Let us prove that if $(E_n)_n$ is a sequence with $\F_{1,\widetilde{Q}_n}(E_n)\le \F_{1,\widetilde{Q}_n}(A_n)$, where $A_n$ is the optimal (centered) annulus, then for $n$
large enough $E_n=F_n\backslash G_n$ with $F_n$ and $G_n$ simply connected and both converging to the unit ball strongly in $W^{2,2}$.
Indeed, if we choose $R_n>0$ with $R_n^2-1=\lambda_n$  and write $\partial E_n$ as a union of simple closed curves $\Gamma_1^n,\dots,\Gamma_{N_n}^n$, we deduce that 
\begin{equation*}
	 \sum_{i=1}^{N_n} (P+W)(\Gamma_i^n)\leq  \F_{1, \widetilde Q_n}( E_n) \leq \F_{1, \widetilde Q_n }( B_{R_n}\setminus B_1) \leq 8\pi +C(1+\widetilde Q_n)\lambda_n,
\end{equation*}
where in the last inequality we have used that for every set $\Omega$, $V_\alpha(\Omega)\les |\Omega|$.
Since  $\partial B_1$ minimizes $P+W$ among all simple closed curves
(by \cite{BucurHenrot,FeKN16} and the isoperimetric inequality it is minimized by circles and then a simple optimization on the radius gives the minimality of $\partial B_1$),
and since $(1+\widetilde Q_n) \lambda_n\to 0$ as $n\to\infty$ we deduce that for all $n$ sufficiently large
we have $N_n\leq 2$. By \cite[Theorem 1.1]{BucurHenrot} we also have $N_n>1$ since otherwise $W(E_n)$ would blow up.  Hence $N_n=2$ for all $n$ sufficiently large and
$\lim_{n\to\infty}W(\Gamma_i^n)= 2\pi$ for $i=1,2$ so that  the  claim follows from Lemma \ref{lem:quantWill2}.

Arguing then exactly as in the proof of Theorem \ref{theo:stabannulus} and using that $\widetilde Q_n\le Q_2$ we obtain that
for $n$ sufficiently large $E_n$ is an annulus if  $Q_2$ is sufficiently small.
\end{proof}
\begin{remark}
 We do not expect that the condition $Q\le c_0 \lambda^{\frac{3+\alpha}{2}}$ is sharp. Indeed, looking at the proofs of Proposition \ref{prop:stabannulus2}
 and Theorem \ref{theo:stabannulus} we see that we have argued separately that $F_n$ and $G_n$ should be balls without exploiting the fact that the volume
 of $E_n=F_n\backslash G_n$ is small. One could hope  to improve the result by obtaining a better control on $V_\alpha(B_{R_n}\backslash B_{r_n})-V_\alpha(E_n)$. 
\end{remark}

As a corollary, we obtain by \eqref{eq:rescale-FlambdaQ} the minimality of annuli for small volumes.
\begin{corollary}\label{corcor}
 There exist $\lambda_1$ and $c_1$ such that   for $m\le \frac{\lambda_1}{\lambda}$ and  $Q^{\frac{2}{3+\alpha}}\le c_1 \lambda$, minimizers of  $\F_{\lambda,Q}$ in $\M(m)$
are  centered annuli.
\end{corollary}

\subsection{Non-existence of minimizers for large charge}
We now prove a  non-existence result for $\alpha\in(1,2)$ and $Q$ large enough (depending on $\lambda$). The restriction $\alpha\in(1,2)$ comes
from the fact that contrarily to what happens for the generalized Ohta-Kawasaki model (\eqref{eq:fun-pre} with $\mu=0$), we cannot easily use a cutting argument. That procedure has roughly the effect of replacing $\alpha$ by $\alpha+1$ (see \cite{KnMu13,KnMu14,FrKN16})
and thus allows to extend the non-existence result from $\alpha \in(1,2)$ to $\alpha\in(0,2)$.\\
For any $Q>0$ we first observe that if a minimizer exists then it must be connected. The following lower bound will thus by useful to prove non-existence results.
\begin{lemma}
There exists $c_\alpha>0$ such that, 
for every $\lambda,Q>0$ and every connected set $E\in \M(|B_1|)$, there holds
 \begin{equation}\label{lowerconnect}
  \F_{\lambda,Q}(E)\geq 
  c_\alpha\lambda^{\frac{2-\alpha}{3-\alpha}} Q^{\frac{1}{3-\alpha}}.
 \end{equation}
\end{lemma}
\begin{proof}
Let $E$ be a connected set and let $d=diam(E)\geq 2$ be its diameter.  If we write $E=F\backslash \cup_{i=1}^n G_i$ with $F$ and $G_i$ simply connected then $diam(E)=diam(F)$ and $W(E)\ge W(F)$.
By \cite{BucurHenrot} we have  $\diam(F) W(F)\geq 4\pi$ so that $d W(E)\ge 4\pi$. Moreover, $P(E)\ge 2d$ and $V_\alpha(E)\ge d^{-2+\alpha}|E|^2$  and therefore
\begin{equation*}
  \F_{\lambda,Q}(E) \geq 2\lambda d+ \frac{4\pi}{d}+Q d^{-2+\alpha}\pi^2.
 \end{equation*}
Since 
\begin{equation*}
   \min_{d\geq 2} \,  \lambda d+Qd^{-2+\alpha}\geq c_\alpha \lambda^{\frac{2-\alpha}{3-\alpha}} Q^{\frac{1}{3-\alpha}},
\end{equation*}
we get \eqref{lowerconnect}.
\end{proof}

We will also need an estimate for the interaction energy of annuli. Since we will also use it in Section \ref{sec:3d} in dimension $3$, we state it in arbitrary dimension.
\begin{lemma}\label{lem:annulus}
Consider $0<\eps\leq \frac{1}{2}$ and $\alpha\in(0,d)$. There exists $C_\alpha>0$ (depending implicitly also on the dimension $d$) such that
\begin{equation}\label{eq:estimValpha}
	V_\alpha(B_1\setminus B_{1-\eps}) \leq C_\alpha
	\begin{cases}
		\eps^2  \quad&\text{ if }1<\alpha< d,\\
		\eps^2 |\ln \eps| \quad&\text{ if }\alpha=1,\\
		 \eps^{1+\alpha}\quad&\text{ if }0<\alpha< 1.
	\end{cases}
\end{equation}
\end{lemma}
\begin{proof}
Let $E=B_1\setminus B_{1-\eps}$ and consider for an arbitrary $x\in E$ the potential
\begin{equation*}
	v_\alpha(x)=\int_E \frac{1}{|x-y|^{d-\alpha}}\,dy = \int_0^\infty \Ha^{d-1}(\partial B_\varrho(x)\cap E) \varrho^{-d+\alpha}\, d\varrho.
\end{equation*}
For any $0<\varrho\le \eps$ we have $\Ha^{d-1}(\partial B_\varrho(x)\cap E)\les \varrho^{d-1}$, hence
\begin{equation*}
	\int_0^\eps \Ha^{d-1}(\partial B_\varrho(x)\cap E) \varrho^{-d+\alpha}\, d\varrho \les  \int_0^\eps \varrho^{\alpha-1} =C_\alpha \eps^\alpha.
\end{equation*}
For $\eps\le \varrho\le 1/2$ we have $\Ha^{d-1}(\partial B_\varrho(x)\cap E)\leq C\eps \rho^{d-2}$, hence
\begin{align*}
	&\int_\eps^{1/2} \Ha^{d-1}(\partial B_\varrho(x)\cap E) \varrho^{-d+\alpha} \,d\varrho \\
	&\qquad \qquad \leq  C\eps \int_\eps^{1/2}\varrho^{-2+\alpha} \,d\varrho \,
	\leq\, C_\alpha
	\begin{cases}
		 \eps \quad &\text{ for }1<\alpha<d,\\
		\eps |\ln \eps| \quad &\text{ for }\alpha=1,\\
	  \eps^{\alpha} \quad &\text{ for }0<\alpha<1.
	\end{cases}
\end{align*}
Finally, for $\frac{1}{2}<\varrho<\infty$ we obtain
\begin{equation*}
	 \int_{1/2}^\infty \Ha^{d-1}(\partial B_\varrho(x)\cap E) \varrho^{-d+\alpha} d\varrho \leq  2^{d-\alpha}\int_{1/2}^\infty \Ha^{d-1}(\partial B_\varrho(x)\cap E) \,d\varrho \leq C_\alpha |E|\leq C_\alpha \eps.
\end{equation*}
From these inequalities and $V_\alpha(B_1\setminus B_{1-\eps}) = \int_E v_\alpha(x)\,dx$ the claim follows.
\end{proof}

We next state our main non-existence result.
\begin{proposition}\label{prop:non-existence}
For any $\alpha\in(1,2)$, 
there exists $Q_3(\alpha)>0$ such that, for all $\lambda,Q>0$ with $Q\geq Q_3(\alpha)(\lambda+\lambda^{\frac{\alpha-1}{2}})$, the functional $\F_{\lambda,Q}$
does not admit minimizers in $\M(|B_1|)$.
\end{proposition}
 
\begin{proof}
Let us start by the case $\lambda\geq 1$. 
If a minimizer exists then by \eqref{lowerconnect},
 \begin{equation}\label{lowerconnect2}
  \min_{\M(|B_1|)} \F_{\lambda,Q}(E)\geq  c_\alpha\lambda^{\frac{2-\alpha}{3-\alpha}} Q^{\frac{1}{3-\alpha}}.
 \end{equation}
Consider as a competitor $N\geq 2$ identical annuli of outer diameter $2$. We may assume that they are so far apart that the interaction energy between different annuli becomes negligible. 
The inner radius is given by $ r_N=\lt(1-\frac{1}{N}\rt)^{1/2}$, and the elastica energy of the competitor is not larger than $10\pi N\les 4\pi \lambda N$ so that the perimeter is dominant. 
By Lemma \ref{lem:annulus} the interaction energy of a single annulus is estimated by a constant times $N^{-2}$. Using this competitor in \eqref{lowerconnect2}, we obtain
\begin{equation*}
   C_\alpha\lt(\lambda N +  \frac{Q}{N }\rt)\geq \min_{\M(|B_1|)} \F_{\lambda,Q}(E)\geq c_\alpha\lambda^{\frac{2-\alpha}{3-\alpha}} Q^{\frac{1}{3-\alpha}}.
\end{equation*}
Optimizing in $N$ we find $N\sim Q^{1/2}\lambda^{-1/2}$ and 
\begin{equation*}
  \lambda^{1/2} Q^{1/2} \geq c_\alpha\lambda^{\frac{2-\alpha}{3-\alpha}} Q^{\frac{1}{3-\alpha}},
\end{equation*}
which leads to a contradiction if  $Q\geq Q_3(\alpha)\lambda$ with $Q_3(\alpha)$ chosen  large enough.

We now consider the case $\lambda\le 1$. As above, if a minimizer exists then  \eqref{lowerconnect2} holds. We now consider a competitor $E_{N,R}$ given by $N\geq 2$ identical annuli
of outer radius $R\geq 2$, to be optimized. We prescribe
\begin{equation}\label{hyp:diam}
 R\geq \lambda^{-1/2},
\end{equation}
so that we are still in the regime where $W(E_{N,R})\les \lambda P(E_{N,R})$. If $R-\eps$ is the inner radius of the annulus, then $\eps\les \frac{1}{RN}$ and using Lemma \ref{lem:annulus} we deduce that 
\begin{equation*}
	V_\alpha(B_R\setminus B_{R-\eps}) = R^{-2+\alpha}V_\alpha(B_1\setminus B_{1-\eps/R}) \leq C_\alpha R^{-2+\alpha}N^{-2}.
\end{equation*} 
Using $E_{N,R}$ in \eqref{lowerconnect2} we get
\begin{equation*}
   C_\alpha\lt(N\lambda R+  \frac{Q}{N} R^{-2+\alpha}\rt) \geq c_\alpha\lambda^{\frac{2-\alpha}{3-\alpha}} Q^{\frac{1}{3-\alpha}}.
\end{equation*}
Optimizing the left-hand side in $R$ yields $R=C_\alpha\lt(\frac{Q}{N^2\lambda}\rt)^{\frac{1}{3-\alpha}}$ and then
\begin{equation*}
   Q^{\frac{1}{3-\alpha}}N^{\frac{1-\alpha}{3-\alpha}}\lambda^{\frac{2-\alpha}{3-\alpha}}  \geq c_\alpha \lambda^{\frac{2-\alpha}{3-\alpha}} Q^{\frac{1}{3-\alpha}},
\end{equation*}
which gives a contradiction for $N\geq N_*(\alpha)$ sufficiently large. It only remains to check that such a choice is compatible with hypothesis \eqref{hyp:diam}. The latter is equivalent to
\begin{equation*}
 Q\lambda^{-\frac{\alpha-1}{2}}\geq N_*^2(\alpha),
\end{equation*}
which is satisfied if $Q\geq Q_3(\alpha)\lambda^{\frac{\alpha-1}{2}}$ for $Q_3(\alpha)$ chosen sufficiently large.
\end{proof}
 
As a consequence, we obtain the following non-existence result for large masses from \eqref{eq:rescale-FlambdaQ}.

\begin{proposition}
For every $\alpha\in (1,2)$, there exists $C_\alpha>0$ such that if $m\ge C_\alpha \lambda^{-1}$ and $Qm^{\frac{1+\alpha}{2}}\ge C_\alpha \lambda$, there are no minimizers of $\F_{\lambda,Q}$ in $\M(m)$.
\end{proposition}

%

\section{The three-dimensional case}\label{sec:3d}

In the three dimensional case, the Willmore energy is invariant under dilations. Using Gauss-Bonnet formula 
it is easy to see that it is minimized by balls \cite{Will65}. Moreover, sharp quantitative bounds have been obtained in \cite{DeMu05}, see also \cite{LaSc14}. Since the Willmore energy in dimension three
does not have the same degeneracy as in dimension two and already exhibits a preference for balls, we mainly study the minimization problem \eqref{problem2d-gen} without a perimeter penalization.  
We thus restrict ourselves to the case $\lambda=0$, i.e.~to 
the functional
\begin{align*}
	\F_Q(E)\,=\,  \W( E) + QV_\alpha(E),
\end{align*}
and consider the constrained minimization problem $\inf_{\M(m)} \F_Q(E)$.
By a rescaling, the above minimization problem can be reduced to 
\begin{equation}\label{eq:def-min-cs}
	\inf_{\M(|B_1|)} \F_Q(E).
\end{equation}

\subsection{Upper bound on the energy}
We start by proving a universal upper bound on the energy.
\begin{proposition}\label{prop:upper-bound-Mm}
For any $Q>0$ we have
\begin{align*}
	\inf_{\M(|B_1|)} \F_Q(E)\,\leq\, 8\pi.
\end{align*}
\end{proposition}
\begin{proof}
We show that an annulus with very big radius and volume $|B_1|$ has an energy arbitrarily close to $8\pi$. Since the Willmore energy of any annulus is
$8\pi$ we need to show that the Riesz interaction energy of the annulus vanishes as the radius goes to infinity.

For an arbitrary $R>0$ consider $\delta=\delta(R)$ such that the annulus $A_{R,R+\delta}=B_{R+\delta}(0)\setminus B_R(0)$ has volume $|B_1|$, that is $\delta=\frac{1}{3 R^2}+O(R^{-5})$. 
Applying \eqref{eq:estimValpha} to $\eps=\frac{\delta}{R+\delta}$ we deduce that for $0<\alpha<1$
\begin{align*}
	V_\alpha(A_{R,R+\delta})\,&=\, (R+\delta)^{3+\alpha} V_\alpha(B_1\setminus B_{1-\eps})\\
	&\leq\, C_\alpha R^{3+\alpha}\eps^{1+\alpha}\,
	\,\leq\, C_\alpha R^2\delta^{1+\alpha}\,\leq\, C_\alpha R^{-2\alpha}\,\stackrel{R\to \infty}{\to}\, 0.
\end{align*}
We conclude similarly in the case $1\leq \alpha <3$.
\end{proof}
Let us prove that the same upper bound  can be reached in the class of topological spheres.
\begin{proposition}\label{prop:upper-bound-Mm-sphere}
For any $Q>0$ we have
\begin{align*}
	\inf\big\{\F_Q(E)\,:\, E\in\M(|B_1|), \genus(\partial E)=0\big\}\,\leq\, 8\pi.
\end{align*}
\end{proposition}
\begin{proof}
The construction in \cite{MueRoe14} provides a sequence of sets $E_n\in \M$ with $\partial E_n$ of genus zero and real numbers $r_n$ with $r_n\nearrow 1$ as $n\to\infty$, 
with the following properties:
\begin{equation*}
	E_n\subset A_{r_n,1},\qquad
	|E_n| > \frac{1}{2}|A_{r_n,1}|\qquad \textrm{and} \qquad
	\W(E_n) \to 8\pi.
\end{equation*}
Considering the rescaled sets $\widetilde E_n= \eta_n E_n$, $\eta_n^3= \frac{|B_1|}{|E_n|}$ we observe that $|\widetilde E_n|=|B_1|$ and $\eta_n\to\infty$, $\W(\widetilde E_n)\to 8\pi$. 
Moreover, $\widetilde E_n\subset A_{\eta_nr_n,\eta_n}$ with  $\eta_nr_n\to\infty$ as $n\to\infty$ and $|A_{\eta_nr_n,\eta_n}|\leq 2|B_1|$.
Following the proof of Proposition \ref{prop:upper-bound-Mm} we therefore deduce that 
\begin{equation*}
	V_\alpha(\widetilde E_n)\leq V_\alpha(A_{\eta_nr_n,\eta_n})\stackrel{n\to\infty}{\to} 0.
\end{equation*}
\end{proof}
%
\subsection{Area and diameter bounds}
In this section we prove uniform  area and diameter bounds in the class of sets with Willmore energy strictly below $8\pi$. 

\begin{proposition}\label{prop:bounds}
Let  $E\in\M(|B_1|)$ be such that for  some $\delta>0$, we have
\begin{align*}
	\W(E)\,\leq\, 8\pi-\delta.
\end{align*}
Then, $\partial E$ is connected and there exists a constant $C_\delta$ such that
\begin{align}
	 \Ha^2(\partial E) +\diam(E)^2 \,\leq\, C_\delta. \label{eq:area-estimate}
\end{align}
\end{proposition}
If one restricts to boundaries of genus zero the proposition follows from \cite{Schy12}. We will prove Proposition \ref{prop:bounds} below and first prepare the compactness argument that we will employ. Therefore, 
we need to characterize limits of boundaries of sets with vanishing volume and uniformly bounded perimeter and Willmore energy. 
One natural approach to obtain suitable compactness properties uses the concept of varifolds, which we now quickly introduce. We refer to \cite{Simo83} for a detailed exposition. Since in the following lemma the dimension plays no role, we decided to give the statement in arbitrary dimension $d$. 
We let $G(d,d-1)$ be the Grassmannian manifold of unoriented $(d-1)$ planes in $\R^d$
and say that a Radon measure $\mu$  on $\R^d\times G(d,d-1)$ is an integer rectifiable varifold%
\footnote{more precisely a $(d-1)$-integer rectifiable varifold, but here we restrict ourselves to the co-dimension one case, and simplify the notation.} 
if there exist a countably $(d-1)-$ rectifiable set $\Sigma$, and a function $\theta:\Sigma\to \N$ such that 
$\mu=\theta \H^{d-1}\mres \Sigma\otimes \delta_{T\Sigma}$, i.e.,~for every $\psi\in C^0_c(\R^d\times G(d,d-1))$ there holds
\[
 \int_{\R^d\times G(d,d-1)} \psi d\mu=\int_{\Sigma} \psi(x,T_x\Sigma) \theta(x) d\H^{d-1}.
\]
By a slight abuse of notation, we will consider such a $\mu$ as a measure on $\R^d$ i.e.~we identify it with $\theta \H^{d-1}\mres \Sigma$. We say that $\mu$  has a \emph{weak mean curvature vector}
$H\in L^1_{\loc}(\mu,\,\R^d)$ if for all $\eta\in C^1_c(\R^d,\,\R^d)$ the classical first
variation formula for smooth surfaces generalizes to
\begin{gather*}
  \int_{\R^d} \dive_{\Sigma}\eta \,d\mu\,=\, -\int_{\R^d} H \cdot\eta \,d\mu.
\end{gather*}
We extend the Willmore functional to the set of integer
rectifiable varifolds with weak mean curvature $H$, by setting
\begin{equation*}
	W(\mu)\,=\, \frac{1}{4}\int_{\R^d} |H|^2\,d\mu. 
\end{equation*}
We also recall the notion of oriented integral varifolds, introduced by Hutchinson \cite{Hutc86}. 
An oriented integral varifold $V^o$ is a Radon measure on the product $\R^d\times S^{d-1}$ which satisfies for every $\psi\in C^0_c(\R^d\times S^{d-1})$
\begin{equation*}
	V^o(\psi) \,=\, \int_{\Sigma} \big( \theta^+(x) \psi(x, \nu(x)) + \theta^-(x) \psi(x, -\nu(x))\,d\Ha^{d-1}(x),
\end{equation*}
where $\Sigma$ is a  countably $(d-1)$-rectifiable  set and $\theta^\pm:\Sigma\to\N$ are such that  $ \theta^++\theta^->0$ almost everywhere on $\Sigma$, and where $\nu$ is a unit normal field on $\Sigma$. 
We can naturally associate to $V^o$ the integral varifold $\mu=\theta \Ha^{d-1}\mres \Sigma$ with $\theta=(\theta^++\theta^-)$. 

We may now prove the following compactness result (notice that the novel part is the fact that in the limit, the density $\theta$ is even).
\begin{lemma}\label{lem:even-density}
Consider a sequence $(E_n)_n$ of open, bounded subsets of $\R^d$ with $W^{2,2}$-regular boundaries and inner unit normal field $\nu_n:\partial E_n\to S^{d-1}$. Assume that $|E_n|\to 0$, $\sup_n \Ha^{d-1}(\partial E_n)<\infty$ and $\sup_n W(\partial E_n)<\infty$.\\
Then, there exists a subsequence and an integer rectifiable varifold $\mu$ on $\R^d$ with even density $\theta$ such that $\Ha^{d-1}\mres \partial E_n\rightharpoonup\mu$ as Radon measures. 
Moreover,  $\mu$  has weak mean curvature $H\in L^2(\mu)$ and satisfies
\begin{equation}
	W(\mu) \,\leq\, \liminf_{n\to\infty} \W(\partial E_n). \label{eq:liminf-V}
\end{equation}
\end{lemma}
\begin{proof}
Consider the associated integer rectifiable varifolds $\mu^n=\Ha^{d-1}\mres \partial E_n$, and the oriented varifolds $V_n^o$, defined for $\psi\in C^0_c(\R^d\times S^{d-1})$ as
\begin{align*}
	V_n^o(\psi) \,&=\, \int_{\partial E_n} \psi(x, \nu^n(x))\,d\Ha^{d-1}(x).
\end{align*}
By the bounds on $\H^{d-1}(\partial E_n)$ and $W(E_n)$ we know from Allard's compactness theorem \cite{Simo83} and \cite[Theorem 3.1]{Hutc86} that up to a subsequence, $\mu_n$ and $V_n^o$ converge respectively to an integer rectifiable varifold $\mu$ with weak mean curvature in $L^2(\mu)$
and to an oriented integral varifold $V^0$.  Moreover, \eqref{eq:liminf-V} is satisfied. 
If $V^o$ is represented by $(\Sigma, \theta^\pm,\nu)$, then for any $\psi\in C^0_c(\R^d)$
\begin{align*}
	\mu(\psi) \,=\, \lim_{n\to\infty} \mu_n(\psi) \,&=\, \lim_{n\to\infty} V_n^o(\psi)\\
	&=\, \int\limits_{\R^d\times S^{d-1}} \psi(x) \,dV^o(x,p)\,
	=\, \int_{\Sigma}( \theta^+ + \theta^-)(x) \psi(x)\,d\Ha^{d-1}(x).
\end{align*}
We therefore conclude that $\mu=(\theta^++\theta^-)\H^{d-1}\mres \Sigma$. \\
Now, by definition we have for any $\eta\in C^1_c(\R^d,\R^d)$
\begin{align*}
	 \int\limits_{\R^d\times S^{d-1}} p\cdot \eta(x) \,dV_n^o(x,p)=\, \int_{\partial E_n} \nu^n \cdot\eta\, d\Ha^{d-1} \,=\,-\int_{E_n} \nabla\cdot\eta \,dx. 
\end{align*}
Since $|E_n|\to 0$, passing to the limit in the previous equality yields for $\eta\in C^1_c(\R^d,\R^d)$
\begin{equation*}
 \int\limits_{\R^d\times S^{d-1}} p\cdot \eta(x) \,dV^o(x,p)=0,
\end{equation*}
from which we conclude that  $\theta^+=\theta^-$ and therefore that $\theta=\theta^++\theta^-$is even.
\end{proof}

\begin{proof}[Proof of Proposition \ref{prop:bounds}]
Since the Willmore energy of any compact surface without boundary is at least $4\pi$ we deduce that $\partial E$ is connected.
It remains to prove the required bounds on  the perimeter and diameter.
We first prove  the uniform area bound. Arguing by contradiction we assume that there exists a sequence $(E_n)_n$ in $\M(|B_1|)$ such that
\begin{equation*}
	\lim_{n\to \infty} \Ha^2(\partial E_n)= \infty,\qquad \textrm{and } \quad \W(E_n)\,\leq\, 8\pi-\delta \quad\text{ for all }n\in\N.
\end{equation*}
Set $\widetilde{E}_n= \Ha^2(\partial \widetilde E_n)^{-\frac{1}{2}} E_n$, to obtain a sequence $(\widetilde E_n)_n$ such that 
\begin{equation*}
	\Ha^2(\partial \widetilde{E}_n)\,=\,1,\, \qquad	 \W(\widetilde{E}_n)\,\leq\, 8\pi-\delta\quad \text{for every }n\in\N,\qquad \textrm{and } \qquad \lim_{n\to \infty} |\widetilde E_n|= 0.
\end{equation*}
Applying Lemma \ref{lem:even-density} we deduce that there is a non trivial (since $\mu(\R^3)=1$) limit  integer rectifiable varifold $\mu$ with even density $\theta$  and such that $W(\mu)<8\pi$. 
This contradicts Li-Yau inequality (see \cite[(A.17)]{KuSc04}) and proves the uniform area bound. Finally, by \cite[Lemma 1.1]{Simo93} (see \cite{topping} for the optimal constant) we obtain the diameter estimates
\begin{align}
	\Ha^2(\partial E)^{\frac{1}{2}}W(E)^{-\frac{1}{2}}\,\leq\, \diam(E)\,\leq\, \frac{2}{\pi}\Ha^2(\partial E)^{\frac{1}{2}}W(E)^{\frac{1}{2}},  \label{eq:diam}
\end{align}
and \eqref{eq:area-estimate} follows.
\end{proof}

\begin{remark}
Let $(E_n)_n$ be a minimizing sequence in $\M(|B_1|)$ with $\F_Q(E_n)\leq 8\pi-\delta$. By Proposition \ref{prop:bounds} we obtain that $(E_n)_n$ has uniformly bounded surface area, volume,  Willmore energy, and diameter.
After possibly shifting $E_n$ we obtain a subsequence and a bounded set $E\in\R^3$ of finite perimeter such that $|E|=|B_1|$,
with $|\nabla\Chi_{E}|(\R^3)\leq  C_\delta$ and such that $\Chi_{E_n}\to\Chi_{E}$ in $L^1(\R^3)$, which in particular implies $V_\alpha(E_n)\to V_\alpha(E)$. Moreover, building on ideas of \cite{Simo93}, one can show as in \cite[p.905,p.907]{Schy12} (in particular exploiting the monotonicity formula \cite[Theorem 3]{Schy12}, see also \cite[Section 2.1]{KuSc12}) that $\mu=|\nabla\Chi_{E}|$ is an integer rectifiable 
varifold with mean curvature in $L^2$
and generalized Willmore energy
\begin{align*}
	\W(\mu) \,\leq\, \liminf_{k\to\infty} \W( E_n).
\end{align*}
In order to obtain existence of a minimizer in $\M(|B_1|)$ it remains to prove further regularity properties of $\partial E$ (for example by adapting the arguments in \cite{Schy12}). This goes however beyond the scope of this work.
\end{remark}
%
\subsection{Minimality of the ball for small charge}
We now prove that for $\alpha\in(1,3)$, minimizers of \eqref{eq:def-min-cs} are balls for small $Q$.    
\begin{theorem}\label{teomin}
For any $\alpha\in(1,3)$ there exists $Q_4>0$ such that for every $0<Q<Q_4$ the only minimizers of $\F_Q$ in $\M(|B_1|)$ are balls.
\end{theorem}

\begin{proof}
For the energy of the unit ball we compute 
\begin{align*}
	\F_Q(B_1) \,=\, 4\pi + C_\alpha Q  \,\leq\, 6\pi\quad\text{ for } Q \leq Q_4= \frac{2\pi}{C_\alpha}.
\end{align*}
Let $E\in\M(|B_1|)$ be such that 
\begin{align*}
	\F_Q(E)\leq \F_Q(B_1). 
\end{align*}
This  implies in particular that 
\begin{align}
	\W( E)-\W(B_1)\,\leq\, Q \big( V_\alpha(B_1)- V_\alpha(E)\big). \label{eq:optB-1a}
\end{align}
Since $\W(E)\leq \F_Q(B_1)\leq 6\pi<2\pi^2$ by \cite[Theorem A]{MaNe14}, $\partial E$ is of sphere type. By Proposition \ref{prop:bounds} the diameter of $E$ and the surface area of $\partial E$ are bounded independently of $Q$.
Furthermore, by \cite{DeMu05}, up to a translation there exists a $W^{2,2}$-parametrization $\Psi:\partial B_r\to\R^3$ of $\partial E$ over $\partial B_r$, where $r\geq 1$ is 
chosen such that $\Ha^2(\partial B_r)=\Ha^2(\partial E)$. Thanks to the uniform bound on $\Ha^2(\partial E)$, \cite{DeMu05} yields
\begin{align*}
	\|\Psi-\id\|_{W^{2,2}}\,\les\, \big(\W( E)-4\pi\big)^{\frac{1}{2}}\,\stackrel{\eqref{eq:optB-1a}}{\les} \sqrt{Q}\big( V_\alpha(B_1)-V_\alpha(E)\big)^{\frac{1}{2}}.
\end{align*}
By Sobolev embedding we also have that
\begin{align}
	\omega\,=\, \|\Psi-\id\|_{C^0}\,\les\, \sqrt{Q}\big( V_\alpha(B_1)-V_\alpha(E) \big)^{\frac{1}{2}} \label{eq:optB-6}
\end{align}
becomes arbitrarily small as $Q$ tends to zero. Furthermore, by \cite{RoeSc12} and $|E|=|B_1|$,
\begin{align*}
	\Ha^2(\partial E)-4\pi \,\les\, \big(\W( E)-4\pi\big), 
\end{align*}
which implies
\begin{align}
	r-1\,\les\, \big(\W( E)-4\pi\big)\,\les\, Q\big(V_\alpha(B_1)-V_\alpha(E)\big). \label{eq:optB-4}
\end{align}

Let $v_{\alpha}(x)= \displaystyle\int_{B_1}\frac{dy}{|x-y|^{3-\alpha}}$ denotes the potential of the unit ball. Applying \cite[Lemma 4.5]{KnMu14} we get that for every $c\in \R$,
\begin{align*}
	V_\alpha(B_1) - V_\alpha(E) \,&\leq\, 2   \int_{B_1\backslash E} (v_{\alpha}-c)  -2\int_{E\backslash B_1}(v_{\alpha}-c)\, \,=\,2   \int_{\R^3} (\Chi_{B_1} - \Chi_{ E} )(x)(v_{\alpha} -c)\, . \notag 
\end{align*}
Using that $v_\alpha$ is radially symmetric, we can choose $c=v_\alpha\lt(\frac{x}{|x|}\rt)$, which by Lipschitz continuity of $v_\alpha$ (see \cite[Lemma 4.4]{KnMu14}) gives
\begin{equation*}V_\alpha(B_1) - V_\alpha(E) \leq C_\alpha \int_{E\Delta B_1} ||x|-1| 
	\leq\, C_\alpha \int_{B_{r+\omega}\backslash B_{1-\omega}} \big||x|-1\big|\,,
\end{equation*}
where we have used that since $r\ge 1$, $ B_{1-\omega}\subset E\subset B_{r+\omega}$.
Using \eqref{eq:optB-6} and \eqref{eq:optB-4}, this yields for $Q$ small
\begin{align*}
	V_\alpha(B_1)-V_\alpha(E) \,\leq\, C_\alpha (r-1+\omega)^2\leq C_\alpha Q\big(V_\alpha(B_1)-V_\alpha(E)\big),
\end{align*}
which implies that if $Q$ is small enough then $V_\alpha(E)=V_\alpha(B_1)$ and therefore $E=B_1$.
\end{proof}

\begin{remark}
 In light of \cite[Theorem 1.3]{F2M3}, we expect this result to hold also for $\alpha\in(0,1]$.
However, since our proof relies on the rigidity estimate \cite[Theorem 1.1]{DeMu05} we must work with
sets which are parameterized by a small $W^{2,2}$ function on the sphere rather than with nearly spherical sets. 
For such sets it is unclear how to obtain a Taylor expansion of $V_\alpha$ analogous to \cite[Lemma 5.3]{F2M3}. To overcome this issue, we used that for $\alpha\in(1,3)$
the potential $v_\alpha$ is Lipschitz continuous. This is not the case for $\alpha\in(0,1]$ (see \cite{KnMu14}). 
Of course, this problem would be solved if one could prove an improved convergence theorem for minimizers (or almost minimizers) of \eqref{eq:def-min-cs}.
\end{remark}

\begin{remark}
As a consequence of Theorem \ref{teomin} and the isoperimetric inequality, balls are the only minimizers 
of $\F_{\lambda,Q}$ in $\M(|B_1|)$
for any $0<Q<Q_4$ and  any $\lambda>0$.
\end{remark}

\begin{remark}
By rescaling it follows from Theorem \ref{teomin}
  that for $Qm^{\frac{3+\alpha}{3}}<Q_4$,   the only minimizers of $\F_Q$ in $\M(m)$ are balls.
\end{remark}

\subsection{Properties of minimizers for large charge}
For large charge, we are not able to prove or disprove that minimizers  of $\F_Q$ exist. In the following proposition we just point out that if minimizers exist, they must have more and more degenerate isoperimetric quotient 
as the charge increases. This is somewhat reminiscent of \cite[Theorem 1]{Schy12}.

\begin{proposition}\label{prop:min-large-m}
Assume that there exists a sequence $Q_n\to\infty$ such that for every  $n\in\N$ there exists a set $E_n\in \M(|B_1|)$ with $\F_{Q_n}(E_n)<8\pi$. Then, 
\begin{align*}
	\lim_{n\to \infty}\Ha^2(\partial E_n)\,= \infty \quad\text{ and }\quad \lim_{n\to \infty}\W(E_n)= 8\pi.
\end{align*}
\end{proposition}
\begin{proof}
We have  by the minimality of $E_n$
\begin{align*}
	8\pi -\W(E_n) \,&>\, Q_n V_\alpha(E_n)\\
	&\geq\,  Q_n |B_1|^{2}\diam(E_n)^{-3+\alpha}\\
	\,&\stackrel{\eqref{eq:diam}}{\ges}\,  Q_n \Big(\Ha^2(\partial E_n)^{\frac{1}{2}}\W(E_n)^{\frac{1}{2}}\Big)^{-3+\alpha}
	\,\ges\,   Q_n \Ha^2(\partial E_n)^{\frac{-3+\alpha}{2}} \W(E_n)^{\frac{-3+\alpha}{2}}\\
\end{align*}
This yields, 
\begin{equation*}
	8\pi >  \W(E_n) + C Q_n \Ha^2(\partial E_n)^{\frac{-3+\alpha}{2}} \W(E_n)^{\frac{-3+\alpha}{2}} 
	\geq C_\alpha  Q_n^{\frac{2}{5-\alpha}} \Ha^2(\partial E_n)^{\frac{-3+\alpha}{5-\alpha}},
\end{equation*}
where we have optimized in $\W(E_n)$ in the second line. This gives $\Ha^2(\partial E_n)> C_\alpha Q_n^{\frac{2}{3-\alpha}}\to \infty$. We then conclude by \eqref{eq:area-estimate} that $W(E_n)\to 8\pi$.
\end{proof}

\subsection{A non-existence result}
For the full functional $\F_{\lambda,Q}$, $\lambda>0$ we can prove non-existence for $\alpha\in (2,3)$ and $Q$ large enough.

\begin{proposition}\label{prop:nonexistence-3D}
 For every $\alpha\in (2,3)$, there exists $Q_5(\alpha)$ such that for every $\lambda, Q$ with $Q\ge Q_5(\lambda^{-\frac{3-\alpha}{2}}+\lambda^{\frac{3+\alpha}{2}})$, 
  there is no minimizer of
 $\F_{\lambda,Q}$ in $\M(|B_1|)$. 
\end{proposition}

\begin{proof}
Assume that $Q\gg \lambda^{-\frac{3-\alpha}{2}}+\lambda^{\frac{3+\alpha}{2}}$. 
If a minimizer  $E$ of \eqref{problem2d-gen} exists then it must be connected. Therefore, there exists one connected component $\Sigma$ of $\partial E$ such that $\diam(\Sigma)=\diam(E)$.
By \eqref{eq:diam}, we have 
 \[
  \diam(E)=\diam(\Sigma)\les\sqrt{\H^2(\Sigma) W(\Sigma)}\les \sqrt{P(E) W(E)}.
 \]
Therefore,
\begin{multline}\label{hypabsurd3d}
 \F_{\lambda,Q}(E)= \lambda P(E)+W(E)+Q V_\alpha(E)\ges \sqrt{\lambda} \sqrt{P(E) W(E)} +\frac{Q}{\diam(E)^{3-\alpha}}\\
 \ges \sqrt{\lambda} \diam(E)+\frac{Q}{\diam(E)^{3-\alpha}}\ges 
 \lambda^{\frac{3-\alpha}{2(4-\alpha)}} Q^{\frac{1}{4-\alpha}}.
\end{multline}
For $N\gg1$ and $R$ to be chosen below, consider a competitor $F_{N,R}$ made of $N$ identical annuli of inner radius $R$ and outer radius $R+h$ with $h$ such that each  volume is equal to $N^{-1}$. 
As long as $N^{-1}\ll R^3$, we have $h\simeq R^{-2}N^{-1}$ and $h\ll R$. By \eqref{eq:estimValpha}, we have
\[\F_{\lambda,Q}(F_{N,R})\simeq N\lambda R^2 + N +\frac{Q}{N R^{3-\alpha}}.\] 
We now choose 
\begin{equation*}
 R=\lambda^{-\frac{1}{2}},\quad N=\lambda^{\frac{3-\alpha}{4}}Q^{\frac12}
\end{equation*}
and observe that $N\gg 1$ since by hypothesis $Q\gg \lambda^{-\frac{3-\alpha}{2}}$ and that $N^{-1}\ll R^3$ since $Q\gg \lambda^{\frac{3+\alpha}{2}}$. We then obtain
\[
 \F_{\lambda,Q}(F_{N,R}) \simeq \lt(\frac{\lambda^{3-\alpha}Q^2}{N^{\alpha-1}}\rt)^{\frac{1}{5-\alpha}}+N 
 \simeq\lambda^{\frac{3-\alpha}{4}}Q^{\frac12}
\]
which gives a contradiction to \eqref{hypabsurd3d}  and the fact that $\F_{\lambda,Q}(F_{Q,R})\ge \F_{\lambda,Q}(E)$ since by hypothesis $Q\gg \lambda^{-\frac{3-\alpha}{2}}$ and $\alpha>2$.
\end{proof}

\section*{Acknowledgments}
M. Goldman has been partly supported by the PGMO
research project COCA. M. Novaga was partly supported by the CNR-GNAMPA, 
and by the University of Pisa via grant PRA-2017-23.

\bibliography{RieszWillmore}
\bibliographystyle{plain}

\end{document}